\documentclass[11pt, twoside]{article}

\usepackage{amsmath, amssymb, amsthm}
\usepackage{enumitem}
\usepackage[margin=2.5cm]{geometry}
\usepackage{setspace}

\usepackage{dsfont}
\usepackage{relsize}

\usepackage{verbatim}
\usepackage{cases}

\usepackage{graphicx,float,subcaption}
\usepackage{algorithm}
\usepackage[noend]{algpseudocode}

\newcounter{criteria}

\makeatletter
\newenvironment{algcriteria}[1][htb]{\renewcommand{\ALG@name}{Criterion}\refstepcounter{criteria}\begin{algorithm}[#1]}{\end{algorithm}}
\makeatother

\usepackage{hyperref}
\usepackage[svgnames]{xcolor}

\usepackage[numbers,sort,compress]{natbib}

\newtheorem{theorem}{Theorem}
\newtheorem{corollary}[theorem]{Corollary}
\newtheorem{proposition}[theorem]{Proposition}
\newtheorem{lemma}[theorem]{Lemma}

\theoremstyle{definition}
\newtheorem{assumption}{Assumption}
\newtheorem{definition}{Definition}
\newtheorem{remark}{Remark}

\newcommand{\RR}{\mathbb{R}}

\newcommand{\EE}{\mathbb{E}}

\newcommand{\norm}[1]{\left\lVert #1 \right\rVert}
\newcommand{\inner}[2]{\left\langle #1, #2 \right\rangle}

\newcommand{\tbo}[2]{\begin{bmatrix}#1 \\ #2\end{bmatrix}}

\newcommand{\argmin}{\mathop{\mathrm{argmin}}}
\newcommand{\argmax}{\mathop{\mathrm{argmax}}}

\newcommand{\tr}{\mathrm{tr}}

\newcommand{\opnorm}[1]{\norm{#1}_{\mathrm{op}}}
\newcommand{\hinfnorm}[1]{\norm{#1}_{\mathcal{H}_\infty}}

\title{\textbf{Online Learning for Supervisory Switching Control}}
\author{Haoyuan Sun, Ali Jadbabaie\\ \\Massachusetts Institute of Technology}
\date{}

\begin{document}

\maketitle

\begin{abstract}
    We study supervisory switching control for partially-observed linear dynamical systems.
    The objective is to identify and deploy a suitable controller for the unknown system by periodically selecting among a collection of $N$ candidate controllers, some of which may destabilize the underlying system.
    While classical estimator-based supervisory control guarantees asymptotic stability, it lacks quantitative finite-time performance bounds.
    Conversely, current non-asymptotic methods in both online learning and system identification require restrictive assumptions that are incompatible in a control setting, such as system stability, which preclude testing potentially unstable controllers.
To bridge this gap, we propose a novel, non-asymptotic analysis of supervisory control that adapts multi-armed bandit algorithms to a control-theoretic setting.
    The proposed data-driven algorithm evaluates candidate controllers via scoring criteria that leverage system observability to isolate the effects of state history, enabling both detection of destabilizing controllers and accurate system identification.
    We present two algorithmic variants with dimension-free, finite-time guarantees, where each identifies the matching controller in $O(N \log^2 N)$ steps, while simultaneously achieving finite $L_2$-gain with respect to system disturbances.

\end{abstract}

\renewcommand{\thefootnote}{\roman{footnote}}
\footnotetext[0]{This work was supported by ONR grant \#N00014-23-1-2299 and AFOSR MURI grant \#FA9550-25-1-0375.}
\footnotetext[0]{Corresponding author: Haoyuan Sun (\texttt{haoyuans [at] mit [dot] edu}).}
\renewcommand{\thefootnote}{\arabic{footnote}}

\section{Introduction}
This paper presents a supervisory switching framework that guarantees finite-time adaptation for unknown linear systems, leveraging online algorithm to establish non-asymptotic analysis for this classical problem.
In many applications involving complex systems, such as power systems~\citep{meng2016microgrid}, autonomous vehicles~\citep{aguiar2007trajectory}, and epidemic control~\citep{bin2021hysteresis}, it is impractical for a single controller design to achieve acceptable performance across all operating conditions of the unknown plant.
To address this challenge, \textit{switching control} employs a collection of candidate controllers alongside a switching policy that selects the controller best suited to the current system behavior~\cite{liberzon2003switching}.
This approach has important applications in many established control paradigms, including gain scheduling~\citep{shamma2002gain, rugh2000research} and linear parameter-varying control~\citep{wu1995control,mohammadpour2012control}.

Within this framework, \textit{supervisory switching control} relies on a higher-level policy that monitors system measurements to determine which controller should be inserted into the feedback loop.
One well-studied class of supervisory policies is the \textit{estimator-based supervision}~\citep{morse1996supervisory,morse1997supervisory,hespanha2001tutorial}.
This approach employs a predetermined collection of models that are designed to cover the uncertainties of the unknown plant.
In practice, a finite covering is often used to balance between approximation accuracy and computational~\cite{morse1996supervisory,anderson2000multiple}.
Each model of the collection is associated with a candidate controller providing satisfactory performance for this model.
Then, a multi-estimator is constructed by simulating the predicted output of each model and computing its deviation from the true observed process $y_t$.
A switching signal $p_t$ is then determined by periodically selecting the model most consistent with the observations.
Provided the switches are separated by a \textit{dwell time} sufficiently long to avoid chattering, the certainty equivalence principle ensures the closed-loop system resulted from this policy is asymptotically stable.
However, this traditional estimator-based approach does not quantitatively specify the time required to find a suitable controller and deriving precise finite-time bounds remains an open challenge.

To develop a non-asymptotic analysis of supervisory switching control, we revisit this classical problem through the lens of modern learning techniques.
A natural starting point is the non-asymptotic analysis of system identification, which estimates the parameters of an unknown dynamical system from a single trajectory of input-output data.
Conceptually, estimator-based supervisors select the model that minimizes the error between predicted and observed outputs, akin to the least-squares objectives underpinning system identification.
Recent advances in non-asymptotic system identification have extended classical asymptotic results~\citep{ljung1998system} to quantify the sample complexity of learning accurate system estimates from data~\citep{simchowitz2018learning,ziemann2023tutorial,li2024learning}.
This has led to numerous methods for the non-asymptotic identification of partially observed linear systems, typically by estimating the system's response matrix to an exploratory signal and reconstructing the system parameters via the Ho-Kalman algorithm~\citep{oymak2019non,simchowitz2019learning,sarkar2021finite,bakshi2023new}.
However, these methods cannot be directly applied to our supervisory switching control paradigm because they presuppose system stability.
This assumption is frequently violated in switching systems, where not all candidate controllers stabilize the underlying plant.
Therefore, standard system identification approach cannot ensure safe exploration, as the potentially unstable controllers may cause the system to diverge before sufficient data is collected.

Several works have attempted to address this limitation of performing system identification under instability.
For example, one can explicitly identify system dynamics by injecting large probing signals, and subsequently control the learned system with online convex optimization~\cite{chen2021black}.
A different approach pairs a multi-armed bandit algorithm with a stability certificate to safely switch between candidate controllers for unknown nonlinear systems~\cite{li2023online,kim2025online}.
Crucially, however, these methods apply only to fully-observed systems, requiring direct knowledge of the system states to guarantee both performance and stability.
In contrast, our prior work~\citep{sun2024least} addresses the partially-observed setting by employing a two-step procedure: first detecting and eliminating destabilizing controllers, and then applying least-squares system identification to the remaining stable candidates.
While this approach ensures safety under partial observability, it relies on a conservative instability detection criterion and exhaustively searches through controllers in a fixed order.
This highly inefficient policy may commit to an unstable controller for up to $O(N^2)$ consecutive steps.

\paragraph{Contributions.}
In this work, we propose a data-driven analysis of supervisory switching control that establishes non-asymptotic guarantees for adapting to partially-observed linear systems with a collection of $N$ candidate controllers. 
Drawing on techniques from the machine learning and multi-armed bandit literature, our approach (Algorithm~\ref{alg:meta}) evaluates each candidate controller over fixed switching intervals.
Specifically, our evaluation criteria leverage system observability to approximate state trajectories from output data, effectively isolating the influences of earlier destabilizing controls.
This eliminates the need for the stability or fully-observed state assumptions mandated by previous works.

To achieve this, we introduce two evaluation criteria to simultaneously detect destabilizing controller (Criterion~\ref{alg:instability}) and identify the unknown system (Criterion~\ref{alg:sys-id}).
The fixed interval between switches is deliberately chosen to ensure sufficient measurements for both criteria, thereby providing a quantitative lower bound on the \textit{dwell time} that was only qualitatively defined in classical estimator-based supervision.

Finally, based on these evaluation metrics, our algorithm selects the candidate controller that balances exploration and exploitation.
Unlike standard multi-armed bandit frameworks, which assume independent and stationary rewards, our algorithm explicitly accounts for the evolution of the system states, striking a balance between learning performance and system stability.
We present two algorithmic variants that find the matching controller in $O(N \log^2 N)$ steps with respect to the number of candidate controllers $N$ and at the same time guarantee finite $L_2$ gain with respect to disturbances (Theorem~\ref{thm:ucb-unified}).
These guarantees represent a significant improvement over our prior work~\cite{sun2024least}, which required $O(N^2)$ steps for identification.

\section{Problem Settings}

We consider a collection of discrete-time, partially-observed linear systems $\left\{(C_i, A_i, B_i)\right\}_{i=1}^N$.
The unknown true system parameters $(C_\star, A_\star, B_\star)$ belongs to this collection, and we denote their index as $i_\star$. 
Each of these systems evolves according to the following discrete-time dynamics:
\begin{align*}
    x_{t+1} &= A_i x_t + B_i \bar{u}_t + w_t, \\
    y_t &= C_i x_t + \eta_t,
\end{align*}
where the dimensions are $x_t \in \RR^{d_x}, \bar{u}_t \in \RR^{d_u}$ and $y_t \in \RR^{d_y}$.
We assume zero initialization $x_0 = 0$ and independent Gaussian noise for both the process disturbance $w_t \sim \mathcal{N}(0, \sigma_w^2 I_{d_x \times d_x})$, and observation noise $\eta_t \sim \mathcal{N}(0, \sigma_\eta^2 I_{d_y \times d_y})$.

Each candidate model is paired with a stabilizing candidate linear controller. We assume that while a model's associated controller provides satisfactory performance for that specific model, applying a controller mismatched to the true plant may result in an unstable closed-loop system.

To facilitate switching between candidate controllers, we adopt the \textit{multi-controller} architecture $K(p_t; \check{x}_t, y_t)$ in \citep{hespanha2001tutorial}, where $\check{x}_t$ is the internal state of the controller, $p_t \in [N]$ is a piece-wise constant switching signal indicating the active controller, and $y_t$ is the system's output.
And to ensure persistency of excitation, we inject an independent, additive Gaussian signal $u_t \sim \mathcal{N}(0, \sigma_u^2 I_{d_u \times d_u})$, so that the control signal into the open-loop system is $\bar{u}_t = u_t + K(p_t, \check{x}_t, y_t)$.

For every pair $1 \le i, j \le N$, we denote $M_{[i, j]} = \left(C_{[i, j]}, A_{[i, j]}, B_{[i, j]}\right)$ as the closed-loop system resulting from applying the $j$th candidate controller to the $i$th open-loop model.
As illustrated by Figure~\ref{fig:switching-loop}, applying a constant switching signal $p_t = j$ on the switching system results in a closed-loop modeled by $M_{[i_\star, j]}$.
Furthermore, we can pre-compute every possible closed-loop system $M_{[i,j]}$ from the collection of candidate models and controllers.
Consequently, under a piecewise-constant switching signal, our observations are generated from one of these $N^2$ possible processes.

\begin{figure}[!htb]
    \centering
    \includegraphics{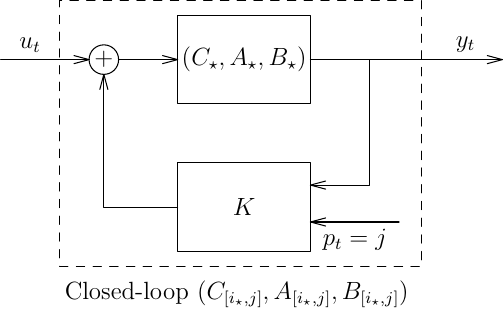}
    \caption{Illustration of a linear switching system (with a fixed switching signal)}
    \label{fig:switching-loop}
\end{figure}

To quantitatively analyze the finite-time performance of this switching system, we assume that applying the matched controller yields an observable system, and that any destabilizing controller can be readily detected:

\begin{assumption}
\label{assum:observability}
    For every $1 \le i \le N$, the correctly matched closed-loop system $M_{[i,i]}$ is $\varepsilon_c$-strictly observable with index $\nu$.
    That is, there exists a constant $\varepsilon_c > 0$ so that
    \[ \sigma_{\min}\left(\left[C_{[i,i]}A_{[i,i]}^{\nu-1}; \dots; C_{[i,i]}A_{[i,i]}; C_{[i,i]}\right]\right) \ge \varepsilon_c \quad \forall \, 1 \le i \le N.\]
\end{assumption}

\begin{assumption}
\label{assum:explosive}
    If a closed-loop system $M_{[i,j]}$ is unstable, there exists a constant $\varepsilon_a > 0$ so that its spectral radius satisfies $\rho(A_{[i,j]}) \ge 1 + \varepsilon_a$.
\end{assumption}

Following the conventions in the learning-for-control literature (e.g. \cite{chen2021black, sarkar2021finite}), our assumptions are non-asymptotic counterparts to the standard control definitions. We explicitly declare the margin bounds $\varepsilon$'s because they are part of our finite-time sample complexity guarantees.
Furthermore, we assume the candidate controllers are sufficiently distinct such that the $N^2$ possible closed-loop systems exhibit different dynamics.

\begin{assumption}
\label{assum:markov-sep}
    Given a system $(C, B, A)$, define its Markov parameter $G$ with horizon $h$ as
    \[G := \left[C A^{h-1} B, \dots, C A B, C B\right].\]
    Then, for a fixed horizon $h$, the Markov parameters $G_{[i,j]}$ corresponding to the $N^2$ possible closed-loop systems are sufficiently different.
    Specifically, there exists a constant $\gamma > 0$ such that for any values of $j$ and $i \neq i'$, we have $\opnorm{G_{[i,j]} - G_{[i',j]}} \ge \gamma$.
\end{assumption}

Intuitively, smaller margins $\varepsilon_a, \varepsilon_c, \gamma$ increase the difficulty of learning the unknown system, requiring longer episode length $\tau$ to reliably evaluate the performance of the candidate controllers.
As shown in Propositions~\ref{thm:instability} and~\ref{thm:sys-id}, our non-asymptotic guarantees scale with $1/\varepsilon_a$, $1/\varepsilon_c$, and $1/\gamma^2$.

\paragraph{Notations.}
For a matrix $M$, we denote $\norm{M}_F$ as its Frobenius norm, $\opnorm{M} = \sigma_{\max}(M)$ as its operator norm, $\rho(M)$ as its spectral radius, and $\tr(M)$ as its trace.
For a stable linear system $(C, A, B)$, we define its $\mathcal{H}$-infinity norm as $\hinfnorm{C, A, B} = \sup_{\norm{s}=1} \sigma_{\max}(C(s I - A)^{-1}B)$, and for simplicity, we use shorthand $\hinfnorm{C, A} = \hinfnorm{C, A, I}$ and $\hinfnorm{A} = \hinfnorm{I, A, I}$.

To streamline our exposition, we occasionally omit constant factors that do not meaningfully impact our conclusions.
We use big-$O$ notation $f \in O(g)$ if $\limsup_{x \to \infty} f(x) / g(x)$ is finite.
Also, we write $f \lesssim g$ if $f \le c \cdot g$ for some universal constant $c$, and $f \asymp g$ if both $f \lesssim g$ and $g \lesssim f$ hold.
Unless stated otherwise, all dimension-dependent terms are written explicitly.
In particular, we consider the Frobenius norm and the trace to be dimension-dependent, whereas the operator norm is dimension-free.

Finally, to analyze stochastic disturbances, we consider a generalization of Gaussian random variables.
Intuitively, a \textit{sub-Gaussian random variable} has tails that is dominated by a Gaussian distribution.
\begin{definition}
    We say that a zero-mean random vector $X \in \RR^d$ is $\sigma^2$-\textit{sub-Gaussian} if for every unit vector $v$ and real value $\lambda$, it holds that
    \[ \EE \exp(\lambda \inner{v}{X}) \le \exp(\sigma^2\lambda^2 / 2).\]
\end{definition} 
\section{Algorithm Design and Evaluation Criteria}
In this section, we introduce our online algorithm for supervisory switching control and establish its finite-time performance guarantees.
The algorithm proceeds in episodes of fixed length $\tau$, where $\tau$ is a constant specified in Propositions~\ref{thm:instability} and~\ref{thm:sys-id} to ensure sufficient measurements are collected for each switching decision.
During each episode, a candidate controller is applied to the plant, and we observe the resulting output trajectory.
Upon completion of the episode, we use the outputs of this episode to evaluate the active controller using a scoring criterion $S = \lfloor \frac{1}{2}(S^{(1)} + S^{(2)}) \rfloor$, where $S^{(1)}$ (Criteria~\ref{alg:instability}) determines whether the controller stabilizes the unknown plant, and $S^{(2)}$ (Criteria~\ref{alg:sys-id}) utilizes least-squares system identification to assess consistency between the controller's associated model and the observed data.
We then select the controller for the next episode by maximizing the average score plus an exploration factor governed by the parameter $a_\ell$. 
After a predetermined number of episodes $L$, we terminate exploration and commit to the most frequently selected controller. This procedure is formalized in Algorithm~\ref{alg:meta}.

\begin{algorithm}
\begin{algorithmic}[1]
    \State Input: collection of candidate models $\{(C_i, A_i, B_i)\}_{i=1}^N$ and their associated controllers.
\State Let $\bar{s}_i(q) = \frac{1}{q} \sum_{r=1}^q s_i[r]$ denote the average score.
    \State Let $Q_i(\ell)$ be the number of episodes where controller $i$ is applied through episode $\ell$.
    \For{episodes $\ell = 1, 2, \dots, L$}
        \If{$\ell \le N$} 
            Select controller $i_\ell = \ell$.
        \Else
            \State Select controller $i_\ell = \argmax_i \bar{s}_i(q_{i, \ell}) + \sqrt{\frac{a_\ell}{q_{i, \ell}}}$, where $q_{i, \ell} = Q_i(\ell-1)$.
        \EndIf
\State Roll out a trajectory of length $\tau$ steps.
        \State Observe outputs $\vec{y}_\ell := (y_{\tau(\ell-1)+1}, \dots, y_{\tau \ell})$.
        \State Compute score $s_{i_\ell}[Q_{i_\ell}(\ell)] = S(\vec{y}_\ell; i_\ell)$.
        \Comment{$S(\cdot; \cdot)$ is determined by Criteria~\ref{alg:instability} and~\ref{alg:sys-id}.}
    \EndFor
    \State Commit to controller $\hat{i} = \argmax_i Q_i(L)$.
\end{algorithmic}
\caption{Algorithm for supervisory switching control}
\label{alg:meta}
\end{algorithm}

We emphasize that this setting introduces fundamental challenges distinct from standard online learning frameworks:
\begin{itemize}[wide, labelindent=3pt]
    \item Standard multi-armed bandit problems assume independent and identically distributed observations, where actions in one round do not affect observations in future rounds.
    In contrast, in our setting, the system state at the end of one episode serves as the initial condition for the next, creating intricate dependencies across switching decisions.
    \item Theoretical analysis of online learning with states --- such as reinforcement learning (RL) --- typically rely on one of two assumptions to bound state trajectories during exploration.
    Episodic RL~\cite{jin2018q} assumes the system is periodically reset to a default initial state distribution, preventing compounding errors across episodes.
    Conversely, non-episodic online RL analyzes continuous operation without resets~\cite{simchowitz2020naive,muehlebach2026the,tian2023can}, but relies on restrictive structural assumptions, such as open-loop system stability, to prevent the state from diverging.
    In contrast, our setting requires exploring potentially destabilizing controllers without resets.
    This means that the systems states may diverge, thus violating the core assumptions of standard RL analysis.
\end{itemize}

To overcome these challenges, our evaluation criteria leverage system observability to isolate the effect of the initial states for each episode.
Specifically, we can use observability to approximate the episode's initial state from the output trajectory and use this estimate to subtract the transient response from the data.
This allows us to evaluate each controller's intrinsic performance without the influence of the state history from previous episodes.

\subsection{Criterion 1: Instability Detection}
\label{sec:instability}
The core objective of this criterion is to construct a statistic that remains bounded when the active controller and its associated model match the true system, but grows unbounded when the active controller is destabilizing.
To achieve this, we leverage the system's observability to reconstruct an estimate of the episode's initial state from output measurements alone.

For convenience, denote the current episode's outputs as $y_1, \dots, y_\tau$ and let $j$ index the active controller.
We write out the following input-output relation for the first $\nu$ steps:
\begin{equation}
\label{equ:obs-dynamics}
\begin{bmatrix} y_\nu \\ \vdots \\y_2 \\ y_1 \end{bmatrix}
=
\mathcal{O}_{[i_\star, j]} x_1
+
\tbo{\mathcal{T}^{(\nu-1)}_{[i_\star, j]}}{\mathbf{0}_{1 \times (\nu-1)}}
\begin{bmatrix} w_{\nu-1} \\ \vdots \\w_2 \\ w_1  \end{bmatrix}
+
\tbo{\mathcal{T}^{(\nu-1)}_{[i_\star, j]}}{\mathbf{0}_{1 \times (\nu-1)}}
\begin{bmatrix} B_{[i_\star, j]}u_{\nu-1} \\ \vdots \\ B_{[i_\star, j]}u_2 \\ B_{[i_\star, j]}u_1  \end{bmatrix}
+
\begin{bmatrix} \eta_\nu \\ \vdots \\ \eta_2 \\ \eta_1  \end{bmatrix},
\end{equation}
where we denote $\mathcal{O} = [CA^{\nu-1}; \dots; CA; C]$ as the observability matrix and
\[\mathcal{T}^{(k)} = \begin{bmatrix} 
    C & CA & CA^2 & \dots & CA^{k-1} \\
    0 & C & CA & \dots & CA^{k-2} \\
    0 & 0 & C & \dots & CA^{k-3}\\
    \vdots & \vdots & \vdots & \ddots & \vdots \\
    0 & 0 & 0 & \dots & C
\end{bmatrix}\]
as a $k \times k$ Toeplitz matrix.
Multiplying both sides of \eqref{equ:obs-dynamics} by the pseudo-inverse $\mathcal{O}_{[i_\star, j]}^\dagger$ yields:\footnote{This pseudo-inverse can be computed efficiently via Arnoldi iteration without explicitly forming the observability matrix.}
\begin{equation}
\label{equ:obs-dynamics-short}
    \mathcal{O}_{[i_\star, j]}^\dagger [y_\nu, \dots, y_1]^\top = x_1 + \mathcal{O}_{[i_\star, j]}^\dagger \xi,
\end{equation}
where $\xi$ aggregates all of the exploratory signal and disturbances.

We first consider the case where the active controller matches the true system ($i_\star = j$).
In this case, the closed-loop dynamics are stable, and the Toeplitz matrices are bounded in norm by $\hinfnorm{C_{[j, j]}, A_{[j, j]}, B_{[j, j]}}$ (see~\cite{tilli1998singular}).
Under Assumption~\ref{assum:observability}, the closed-loop system is $\nu$-strictly observable, which means $\opnorm{\mathcal{O}_{[j, j]}^\dagger} \le 1/\varepsilon_c$.
Consequently, the aggregated disturbance term $\mathcal{O}_{[j, j]}^\dagger \xi$ is sub-Gaussian.
By standard concentration inequalities for sub-Gaussian random variables, the estimate 
\[ \hat{x}_1 := \mathcal{O}_{[j, j]}^\dagger[y_\nu, \dots, y_1]^\top\] 
is a close approximation to the true initial state of the episode $x_1$ with high probability.
Then, simulating a predicted trajectory $\hat{y}_1, \dots \hat{y}_\tau$ starting from $\hat{x}_1$ under the candidate model yields a residual $y_t - \hat{y}_t$ that remains below a threshold $\Theta$ with high probability.

Next, when the active controller is destabilizing, Assumption~\ref{assum:explosive} ensures the closed-loop system possesses an explosive mode in the direction of the leading eigenvector of $A_{[i_\star, j]}$.
Since the system is persistently excited by Gaussian process noise, this explosive mode is activated with non-zero probability.
In this regime, the residual $y_t - \hat{y}_t$ can be decomposed as the sum of estimation error plus accumulated disturbances, where the first part arises because $\hat{x}_1$ does not approximate the true initial state $x_1$ when $i_\star \neq j$.
Critically, the estimation error is determined entirely by the first $\nu$ observations, while the system disturbances are accumulated at every step of the trajectory.
As a result, the accumulated disturbance contains a component driven by process noise realized after the $\nu$-th step, which is independent of the estimation error.
Hence, with non-zero probability, these independent components will not cancel, and contributions from the explosive mode will dominate the residual for sufficiently large $\tau$.
Thus, the residual $y_\tau - \hat{y}_\tau$ will often exceed the threshold $\Theta$ when the active controller is destabilizing.

We formalize this intuition in the following instability criterion, which outputs 1 when it classifies the current closed-loop system as stable and 0 otherwise.
\begin{algcriteria}
\begin{algorithmic}[1]
    \State Input: an observed trajectory $y_1, \dots, y_\tau$ of length $\tau$, active controller $j$.
    \State Let $\mathcal{O}_{[j,j]}$ be the observability matrix of the model $(C_{[j,j]}, A_{[j,j]}, B_{[j,j]})$, compute
    \[\hat{x}_1 = \mathcal{O}_{[j,j]}^\dagger [y_\nu, \dots, y_1]^\top.\]
    \State Compute a predicted trajectory
    \[\hat{y}_{t} = C_{[j,j]} A_{[j,j]}^{t-1} \hat{x}_1.\]
\If{$\norm{y_{\tau} - \hat{y}_{\tau}} > \sqrt{2d_y \Theta\log(2d_y/\delta_1})$} \Return 0
\Else {} \Return 1
    \EndIf
\end{algorithmic}
\caption{Instability detection criterion $S^{(1)}$}
\label{alg:instability}
\end{algcriteria}

We state the theoretical guarantees for this criterion below, with the full proof deferred to Appendix~\ref{sec:proof-instability}.

\begin{proposition}
\label{thm:instability}
    Given a fixed threshold $\Theta$ depending only on $\varepsilon_c, \sigma_u, \sigma_w, \sigma_\eta$ and $\max_i \hinfnorm{C_{[i, i]}, A_{[i, i]}, B_{[i, i]}}$,
    and a sufficiently long trajectory length
    \[\tau \ge T_1 \asymp \nu + \log(1+\varepsilon_a)^{-1} \log\left(\frac{d_y \Theta \log(2d_y/\delta_1)}{\varepsilon_a \varepsilon_c \sigma_w}\right),\]
    the instability criterion $S^{(1)}(y_1, \dots, y_\tau; j)$ satisfies:
    \begin{itemize}[wide, labelindent=3pt]
    \item If the $j$-th controller destabilizes the underlying plant, $S^{(1)} = 0$ with probability $\ge 2/5$.\footnote{We use $2/5$ here to streamline the exposition. A higher probability of detection (up to $1/2$) can be achieved with a longer trajectory length $T_1$.}
    \item If the active controller matches the true system ($i_\star = j$), $S^{(1)} = 1$ with probability $\ge 1 - \delta_1$.
    \end{itemize}
\end{proposition}

\begin{remark}
We note that the episode length $T_1$ is fixed over time. In contrast, the episode length in~\cite{sun2024least} has to increase over time, leading to worse overall sample complexity.
\end{remark}

\begin{remark}
    A sharper, dimension-free guarantee for instability detection can be attained by employing a more sophisticated thresholding scheme.
    As this variant follows directly from the proof of Proposition~\ref{thm:instability}, we defer its discussion to Appendix~\ref{sec:proof-instability}.
\end{remark}

While this criterion may not successfully detect a destabilizing controller in a single episode, the matching controller will achieve a strictly higher expected score over repeated trials.
This separation in expected scores enables the algorithm to reject destabilizing candidates over multiple episodes.
We note, however, that this criterion is inconclusive for the remaining candidate controllers that do not match the true system but are stabilizing.
We shall address this scenario with our second criterion.

\subsection{Criterion 2: System Identification}
\label{sec:sys-id}
Having established a mechanism to detect destabilizing controllers, we now introduce a second criterion to distinguish the matching controller from other stabilizing candidates.
The objective is to construct a statistic that remains small when the active controller's associated model matches the true system, but is large when there is a mismatch.
Like the previous part, we denote the current episode's outputs as $y_1, \dots, y_\tau$ and let $j$ index the active controller.
We begin by recursively expand the system dynamics over a horizon of $h$ steps:
\begin{align*}
    y_t &= C_{[i_\star, j]} x_t + \eta_t \\
    &= C_{[i_\star, j]} (A_{[i_\star, j]} x_{t-1} + B_{[i_\star, j]} u_{t-1} + w_{t-1}) + \eta_t \\
    &= C_{[i_\star, j]} A_{[i_\star, j]}^{h} x_{t-h} \! + \! \underbrace{\sum_{s=1}^h C_{[i_\star, j]} A_{[i_\star, j]}^{s-1} B_{[i_\star, j]} u_{t-s}}_{G_{[i_\star, j]} z_t} \! + \! \sum_{s=1}^h C_{[i_\star, j]} A_{[i_\star, j]}^{s-1} w_{t-s} + \! \eta_t \\
    &:= G_{[i_\star, j]} z_t + e_t,
\end{align*}
where $z_t = [u_{t-h}; \dots; u_{t-1}]$ is a sliding window of exploratory input signals, $e_t$ accounts for both the initial condition and accumulated noise, and $G_{[i_\star, j]}$ denotes the Markov parameter that is equal to the system's output controllability matrix over horizon $h$:
\[G_{[i_\star, j]} := \left[C_{[i_\star, j]} A^{h-1}_{[i_\star, j]} B_{[i_\star, j]}, \dots, C_{[i_\star, j]} A_{[i_\star, j]} B_{[i_\star, j]}, C_{[i_\star, j]} B_{[i_\star, j]}\right].\]
We then construct a shifted output trajectory by subtracting both the predicted trajectory rolled out from the estimated initial condition and the candidate model's nominal response from the Markov parameters:
\begin{align*}
    \widetilde{y}_t 
    = y_t - C_{[j, j]} A_{[j, j]}^{t-1}\hat{x}_1 - G_{[j, j]} z_t
    = (G_{[i_\star, j]} - G_{[j, j]})z_t + (e_t - C_{[j, j]} A_{[j, j]}^{t-1}\hat{x}_1)
    := (G_{[i_\star, j]} - G_{[j, j]}) z_t + r_t,
\end{align*}
where $r_t$ aggregates the residual terms not modeled by $z_t$.
Because $r_t$ only depends on the initial state of the horizon and the disturbances, $r_t$ and $z_t$ are independent.

When the active controller matches the true system ($i_\star = j$), the parameter gap $\Delta := G_{[i_\star, j]} - G_{[j, j]}$ is identically zero. 
With the same stability arguments from Section~\ref{sec:instability}, the residual term $r_t$ is $\sigma_r^2$-sub-Gaussian, where we can show that $r_t \in O(1/\varepsilon_c)$.
Estimating $\Delta$ from the data points $\left\{(z_t, \widetilde{y}_t)\right\}_{t=\nu+h+1}^{\tau}$ via ordinary least-squares (OLS) yields:
\begin{align*}
    \widehat\Delta 
    := \argmin_\Delta \sum_{t=\nu+h+1}^{\tau} \norm{\widetilde{y}_t - \Delta z_t}^2
    = \left(\sum_{t=\nu+h+1}^{\tau} \widetilde{y}_t z_t^\top\right) \left(\sum_{t=\nu+h+1}^{\tau} z_t z_t^\top\right)^{-1}.
\end{align*}
We can further rewrite the OLS estimate as follows:
\begin{equation}
\label{equ:ols-expanded}
\begin{aligned}
    \widehat\Delta 
&= \left(\sum_{t=\nu+h+1}^{\tau} (\Delta z_t + r_t) z_t^\top\right) \left(\sum_{t=\nu+h+1}^{\tau} z_t z_t^\top\right)^{-1} \\
    &= \Delta + \left(\sum_{t=\nu+h+1}^{\tau} r_t z_t^\top\right) \left(\sum_{t=\nu+h+1}^{\tau} z_t z_t^\top\right)^{-1}.
\end{aligned}
\end{equation}
Because $\Delta = 0$ when $i_\star = j$, for sufficiently large $\tau$, the value of $\opnorm{\widehat\Delta}$ falls below a threshold $\gamma$ with high probability.

Next, when the active controller does not match the true system ($i_\star \neq j$), we seek to show that the OLS estimate \eqref{equ:ols-expanded} is large.
In this case, Assumption~\ref{assum:markov-sep} guarantees that $\opnorm{\Delta} \ge \gamma$.
Crucially, in the second term, $r_t$'s and $z_t$'s are independent.
Then, there is non-zero probability that the residuals do not cancel the true difference $\Delta$, yielding $\opnorm{\widehat\Delta} \ge \opnorm{\Delta} \ge \gamma$.

To derive bounds free of the system dimensions $d_x, d_u, d_y$, our criterion evaluates the OLS estimate along specific directions.
Assumption~\ref{assum:markov-sep} implies that, for every $k \neq j$, there exist unit vectors $(u_k, v_k)$ satisfying $u_k^\top (G_{[k, j]} - G_{[j, j]}) v_k = \opnorm{G_{[k, j]} - G_{[j, j]}} \ge \gamma$.
We refer to these as the \textit{critical directions}, which can be pre-computed from the collection of candidate models.

\begin{algcriteria}
\begin{algorithmic}[1]
    \State Input: an observed trajectory $y_1, \dots, y_\tau$ of length $\tau$, active controller $j$.
    \State Let $\mathcal{O}_{[j,j]}$ be the observability matrix of the model $(C_{[j,j]}, A_{[j,j]}, B_{[j,j]})$, compute
    \[\hat{x}_1 = \mathcal{O}_{[j,j]}^\dagger [y_\nu, \dots, y_1]^\top.\]
    \State Compute a predicted trajectory
    \[\widetilde{y}_{t} = y_t - C_{[j,j]} A_{[j,j]}^{t-1} \hat{x}_1 - G_{[j,j]} z_t.\]
    \ForAll{critical directions $(u_k, v_k), k = 1, \dots, N$}
    \State Compute OLS estimate $\widehat{\Delta}_k$ on data points $\left\{(v_k^\top z_t, u_k^\top \widetilde{y}_t)\right\}_{t=\nu+h+1}^\tau$
    \If{$|\widehat{\Delta}_k| > \gamma$} \Return 0
    \EndIf
    \EndFor
    \State \Return 1
\end{algorithmic}
\caption{System identification criterion $S^{(2)}$}
\label{alg:sys-id}
\end{algcriteria}

We state the theoretical guarantees for this criterion below.
And the proof of this result in given in Appendix~\ref{sec:proof-sysid}.

\begin{proposition}
\label{thm:sys-id}
    For a sufficiently long trajectory length
    \[\tau \ge T_2 \asymp \nu + h + \frac{\sigma_r^2}{\sigma_u^2 \gamma^2} \log(N / \delta_2),\]
    and when controller $j$ is not destabilizing the underlying plant, the system identification criterion $S^{(2)}(y_1, \dots, y_\tau; j)$ satisfies:
    \begin{itemize}[wide, labelindent=3pt]
    \item If we have a matching controller ($i_\star = j$), then $S^{(2)} = 1$ with probability $\ge 1 - \delta_2$.
    \item Otherwise, if $i_\star \neq j$, then $S^{(2)} = 0$ with probability $\ge 1/2$.
    \end{itemize}
\end{proposition}

We note that the ratio $\frac{\sigma_r^2}{\sigma_u^2}$ in our bound corresponds to the inverse signal-to-noise ratio of the exploratory signal $u_t$.
Also, the $1/\gamma^2$ term matches the standard sample complexity of estimating unknown parameters in the learning theory literature.
By combining both Criterion~\ref{alg:instability} and Criterion~\ref{alg:sys-id}, the matching controller achieves the highest expected score over repeated episodes, thus provably identify the most suitable controller.

\section{Main Result}
With the guarantees of the two evaluation criteria established, we now analyze the performance of our supervisory switching control algorithm.
To balance exploration and exploitation, we draw inspiration from the Upper Confidence Bound (UCB) algorithm~\cite{auer2002finite,audibert2010best}.
At each episode, we select the controller by maximizing its empirical average score plus an exploration bonus.
The magnitude of this bonus encourages exploration of controllers with limited interactions, and the parameter $a_\ell$ controls the exploration strength.

We emphasize that standard UCB analysis fails in our setting due to the state dependency between different episodes and the requirement for closed-loop system stability.
However, by leveraging the properties of the criterion $S$ (Propositions~\ref{thm:instability} and~\ref{thm:sys-id}), we can isolate the empirical average scores from the initial state of each episode.
This effectively allows for a UCB-style analysis by decoupling the episodes.

We present two distinct choices for the exploration parameter $a_\ell$.
The first choice prioritizes aggressive, uniform exploration, while the second dynamically scales to provide a more refined trade-off between exploration and exploitation.
We demonstrate that both variants enable the framework to identify the matching controller in $O(N \log N)$ episodes while simultaneously ensuring system stability.

\begin{theorem}
\label{thm:ucb-unified}
    Suppose the length of each episode satisfies $\tau \ge \max(T_1, T_2)$, and total number of episodes equal to $L = O(N \log (N/\delta_3))$.
    Let the exploration parameter set to either $a_\ell = \frac{L}{72N}$ or $a_\ell = \frac{1}{2} \log(\frac{\pi^2 N \ell^2}{6\delta_3})$.
    Then with probability $1-\delta_3$, Algorithm~\ref{alg:meta} satisfies the following properties:
    \begin{itemize}[wide, labelindent=3pt]
    \item The algorithm commits to the matching controller at termination, i.e. $i_\star = \hat{i}$.
    \item The closed-loop trajectory has finite $L_2$-gain with respect to disturbances $u_t$ and $w_t$, i.e. there exist constants $C_0, C_1$ so that for any $T > \tau L$,
    \begin{align*} 
        \sum_{t=1}^T \norm{x_t}^2 
        \le C_0 \left(\sum_{t=1}^{\tau L} \norm{u_t}^2 + \norm{w_t}^2\right)
        + C_1 \left(\sum_{t=\tau L + 1}^T \norm{u_t}^2 + \norm{w_t}^2\right).
    \end{align*}
    \end{itemize}
\end{theorem}

\begin{proof}[Proof sketch]
From Propositions~\ref{thm:instability} and~\ref{thm:sys-id}, we can probabilistically bound the outcomes of our evaluation metric $S$ independently of the initial state of the episodes.
This enables us to decouple the state dependencies of different episodes.

Let $\mu^+$ denote a lower bound on the expected score of the matching controller ($j = i_\star$), and let $\mu^-$ denote an upper bound on the expected score of any non-matching controller ($j \neq i_\star$).
Recall that $Q_i(\ell)$ denotes the number of times controller $i$ has been applied through episode $\ell$.
Similar to standard UCB analysis, if Algorithm~\ref{alg:meta} selects a non-matching controller $j \neq i_\star$ at episode $\ell$, then at least one of the following three conditions must be met:
\begin{itemize}[wide, labelindent=3pt]
    \item Average score for the matching controller is too low: $\bar{s}_{i_\star}[Q_{i_\star}(\ell)] \le \mu^+ - \sqrt{\frac{a_\ell}{Q_{i_\star}(\ell)}}$.
    \item Average score for a non-matching controller is too high: $\exists j \neq i_\star$, s.t. $\bar{s}_j[Q_j(\ell)] \ge \mu^- + \sqrt{\frac{a_\ell}{Q_j(\ell)}}$.
    \item The exploration bonus is too big:  $\exists j \neq i_\star$, s.t. $2 \sqrt{\frac{a_\ell}{Q_j(\ell)}} > \mu^+ - \mu^-$.
\end{itemize}
Notice that the third condition establishes an upper bound on the number of times a non-matching controller can be selected before its exploration bonus shrinks below the expected score gap.
Let $\mathcal{E}$ denote the ``good'' event where the first two conditions do not occur.
Conditioned on $\mathcal{E}$, we can upper-bound the frequency of selecting any non-matching controller by $4(N-1)(\mu^+ - \mu^-)^{-2} a_\ell$.
We can then lower bound $\Pr(\mathcal{E})$ with martingale concentration equalities, yielding a probabilistic lower bound on the frequency of selecting the matching controller.

To establish the first part of the theorem, we evaluate our two specific choices for the exploration parameter $a_\ell$:
\begin{itemize}[wide, labelindent=3pt]
    \item In the first variant, we set $a_\ell = \frac{L}{72N}$ to ensure that the matching controller is selected at least half of the times under $\mathcal{E}$. Then we choose a sufficiently long horizon $L$ so that the event $\mathcal{E}$ occurs with probability $1-\delta_3$.
    \item In the second variant, we dynamically set $a_\ell = \frac{1}{2} \log(\frac{\pi^2 N \ell^2}{6\delta_3})$ so that the event $\mathcal{E}$ occurs with probability $1-\delta_3$ for any horizon $L$.
    Then, we pick a sufficiently large $L$ so that the matching controller is selected at least half of the times under $\mathcal{E}$.
\end{itemize}

The second part of the theorem follows from the fact that applying the matching controller results in a stable closed-loop system.
The constant $C_0$ is an upper bound on the worst-case transient energy accumulated over the $O(N \log (N/\delta_3))$ exploratory episodes.
And the constant $C_1$ is directly inherited from performance of the matching controller.

For the complete proof of this result, we refer the readers to Appendix~\ref{sec:proof-ucb}.
\end{proof}

The first part of this theorem establishes a sample complexity bound, stating that our algorithm requires $O(N \log N)$ episodes to identify the matching controller with high probability.
From Propositions~\ref{thm:instability} and~\ref{thm:sys-id}, each episode takes $O(\frac{\log N}{\varepsilon_c\varepsilon_a \gamma^2})$ steps, so that our algorithm takes $O(N \log^2 N)$ steps with respect to $N$.
This represents a significant improvement over our prior work~\citep{sun2024least} that required $O(N^2)$ steps.

The second part of this theorem addresses closed-loop stability.
We note that the transient term $C_0$ scales as $\exp(O(N \log(N/\delta_3)))$, which corresponds to the cumulative state growth incurred while exploring potentially destabilizing controllers over $O(N \log (N/\delta_3))$ episodes.
As shown by~\citet{chen2021black}, such exploration penalty is unavoidable when testing unstable controllers.
Compared to the fully-observed setting in~\citet{li2023online}, our bound contains an additional $\log(N/\delta_3)$ factor in the exponent.
This reflects the reality that our instability criterion cannot guarantee successful detection in every single episode due to partial observation, thereby necessitating repeated trials to achieve high-confidence identification.
Finally, we emphasize that the value of $C_0$ reflects a highly conservative, worst-case scenario where all destabilizing controllers are applied consecutively, which would be \textit{exceedingly rare} in practice because the UCB selection mechanism heavily favors the matching controller.
In contrast, the naive exploration strategy in our previous work~\cite{sun2024least} would commit unstable controllers all in consecutive steps and \textit{always} result in the worse case transient behavior.

\begin{remark}
    The stability guarantees of the classical estimator-based supervision are typically given in terms of a time-discounted $L_2$-norm.
    Specifically, for a forgetting factor $\lambda > 0$, \citet{hespanha2001tutorial} showed that there exist constants $\overline{C}_0, \overline{C}_1$ such that for any $T > 0$,
    \[\sum_{t=1}^T e^{-2\lambda (T-t)}\norm{x_t}^2 \le \overline{C}_0 e^{-2\lambda T} + \overline{C}_1 \sum_{t=1}^T e^{-2\lambda (T-t)} \norm{w_t}^2.\]
    In contrast, our Theorem~\ref{thm:ucb-unified} bounds the undiscounted case ($\lambda = 0$).
    This regime is not covered by classical analysis, as the traditional framework requires a dwell time that scales inversely with $\lambda$.
    Furthermore, we can provide quantitative bounds for the constants $C_0, C_1$ in terms of the problem parameters, whereas classical analysis is asymptotic and leaves the explicit values of $\overline{C}_0$ and $\overline{C}_1$ unspecified.
\end{remark}

\section{Numerical Simulation}
To numerically verify the efficacy of our Algorithm~\ref{alg:meta}, we consider an example studied in~\cite{anderson2000multiple}, where there is a second-order system with an unknown gain $k$:
\[ \left\{\frac{k(s-1)}{(s+1)(s-2)} : k \in [1, 40)\right\}.\]
This family of systems are difficult to control because they are non-minimum phase and due to the large range over the uncertainty.
So, we construct candidate models on the finite-covering $k_i = 1.2^i, i \in \{0, \dots, 20\}$, which is found to closely-approximate the uncertainty on $k$~\cite{anderson2000multiple}.

We simulate a zero-order-hold discretization of this system at 1000 Hz.
To illustrate a \textit{worst-case} scenario, we set a large initial condition sampled from $\mathcal{N}(0, 20)$, the true index to be $i_\star = 20$, and fix the algorithm's starting index to be $i_1 = 0$.
In Figure~\ref{fig:simulation}, we compare the outputs of Algorithm~\ref{alg:meta} and the estimator-based supervision.
The estimator-based supervisor is set with a dwell-time of 0.2 seconds, and forgetting factor $\lambda = 0.04$.
For our algorithm, we set the maximum episode length to be 0.2 seconds and added an early stopping rule to Criteria~\ref{alg:instability}, so we switch as soon as the instability threshold is crossed.
All policy parameters were tuned to give the best empirical behaviors for this scenario.

\begin{figure}[!htb]
    \centering
    \includegraphics[width=0.7\linewidth]{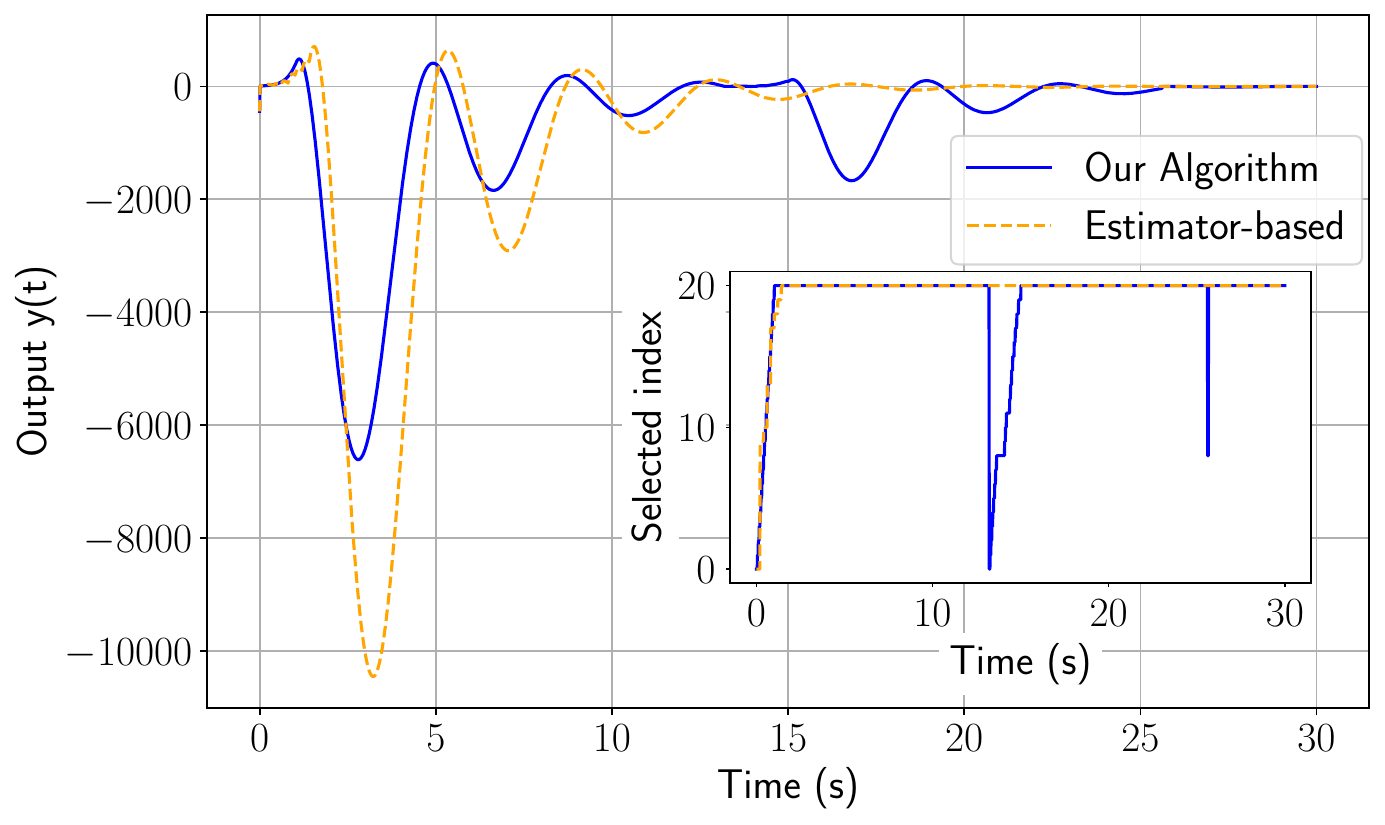}
    \caption{Numerical simulation comparing Algorithm~\ref{alg:meta} and classical estimator-based supervision. The smaller embedded plot shows the index of the candidate controllers selected by each policy.}
    \label{fig:simulation}
\end{figure}

With an exploration factor $\alpha = 0.3 \ell$, our UCB-style selection rule tests different candidates over three distinct periods to keep exploring the system.
In contrast, the estimator-based supervisor is greedy and stops switching once a suitable controller is found.
We observe that our algorithm results in a smaller transient than the estimator-based supervisor.
This is primarily due to the early-stopping rule, where our precise theoretical analysis allows for more granular adjustments of the policy parameters and switching more aggressively when the active controller is destabilizing the underlying plant. 
\section{Conclusion}
In this work, we presented a supervisory switching control framework that guarantees finite-time adaptation for unknown partially-observed linear systems with a collection of candidate controllers.
Our proposed data-driven method achieves a dimension-free sample complexity of $O(N \log^2 N)$ steps to identify the matching controller while maintaining finite $L_2$-gain, even when some of the candidate controllers risk destabilizing the underlying plant.

Looking forward, a promising direction for future research involves expanding the capacity of our evaluation criteria.
Currently, our criteria only evaluate the model associated with the active controller.
Future algorithmic extensions could significantly strengthen both safety and performance guarantees by simultaneously cross-evaluating the observed trajectory against all candidate models within a single episode.

\bibliographystyle{plainnat}
\bibliography{reference}

@inproceedings{hespanha2001tutorial,
  title={Tutorial on supervisory control},
  author={Hespanha, Joao P},
  booktitle={Lecture Notes for the workshop Control using Logic and Switching for the 40th Conf. on Decision and Contr., Orlando, Florida},
  year={2001},
}

@inproceedings{li2023online,
  title={Online switching control with stability and regret guarantees},
  author={Li, Yingying and Preiss, James A and Li, Na and Lin, Yiheng and Wierman, Adam and Shamma, Jeff S},
  booktitle={Learning for Dynamics and Control Conference},
  year={2023},
}

@inproceedings{simchowitz2020naive,
  title={Naive exploration is optimal for online {LQR}},
  author={Simchowitz, Max and Foster, Dylan},
  booktitle={International Conference on Machine Learning},
  year={2020},
}

@article{jin2018q,
  title={Is {Q}-learning provably efficient?},
  author={Jin, Chi and Allen-Zhu, Zeyuan and Bubeck, Sebastien and Jordan, Michael I},
  journal={Advances in neural information processing systems},
  volume={31},
  year={2018}
}

@article{hsu2012tail,
  title={A tail inequality for quadratic forms of subgaussian random vectors},
  author={Hsu, Daniel and Kakade, Sham M and Zhang, Tong},
  journal={Electronic Communications in Probability},
  volume={17},
  year={2012},
  publisher={Institute of Mathematical Statistics}
}

@inproceedings{audibert2010best,
  title={Best arm identification in multi-armed bandits},
  author={Audibert, Jean-Yves and Bubeck, S{\'e}bastien},
  booktitle={Conference on learning theory (COLT)},
  pages={13--p},
  year={2010}
}

@article{auer2002finite,
  title={Finite-time analysis of the multiarmed bandit problem},
  author={Auer, Peter and Cesa-Bianchi, Nicolo and Fischer, Paul},
  journal={Machine learning},
  volume={47},
  number={2},
  pages={235--256},
  year={2002},
  publisher={Springer}
}

@inproceedings{tian2023can,
  title={Can Direct Latent Model Learning Solve Linear Quadratic Gaussian Control?},
  author={Tian, Yi and Zhang, Kaiqing and Tedrake, Russ and Sra, Suvrit},
  booktitle={Learning for Dynamics and Control Conference},
  pages={51--63},
  year={2023},
  organization={PMLR}
}

@inproceedings{simchowitz2018learning,
  title={Learning without mixing: Towards a sharp analysis of linear system identification},
  author={Simchowitz, Max and Mania, Horia and Tu, Stephen and Jordan, Michael I and Recht, Benjamin},
  booktitle={Conference On Learning Theory},
  pages={439--473},
  year={2018},
}

@inproceedings{oymak2019non,
  title={Non-asymptotic identification of {LTI} systems from a single trajectory},
  author={Oymak, Samet and Ozay, Necmiye},
  booktitle={2019 American control conference (ACC)},
  pages={5655--5661},
  year={2019},
  organization={IEEE}
}

@article{abbasi2011improved,
  title={Improved algorithms for linear stochastic bandits},
  author={Abbasi-Yadkori, Yasin and P{\'a}l, D{\'a}vid and Szepesv{\'a}ri, Csaba},
  journal={Advances in neural information processing systems},
  volume={24},
  year={2011}
}

@article{laurent2000adaptive,
  title={Adaptive estimation of a quadratic functional by model selection},
  author={Laurent, Beatrice and Massart, Pascal},
  journal={Annals of statistics},
  pages={1302--1338},
  year={2000},
  publisher={JSTOR}
}

@article{tilli1998singular,
  title={Singular values and eigenvalues of non-Hermitian block Toeplitz matrices},
  author={Tilli, Paolo},
  journal={Linear algebra and its applications},
  volume={272},
  number={1-3},
  pages={59--89},
  year={1998},
  publisher={Elsevier}
}

@article{rudelson2013hanson,
  title={Hanson-Wright inequality and sub-Gaussian concentration},
  author={Rudelson, Mark and Vershynin, Roman},
  year={2013}
}

@book{wainwright2019high,
  title={High-dimensional statistics: A non-asymptotic viewpoint},
  author={Wainwright, Martin J},
  volume={48},
  year={2019},
  publisher={Cambridge university press}
}

@inproceedings{ziemann2023tutorial,
  title={A Tutorial on the Non-Asymptotic Theory of System Identification},
  author={Ziemann, Ingvar and Tsiamis, Anastasios and Lee, Bruce and Jedra, Yassir and Matni, Nikolai and Pappas, George J},
  booktitle={Conference on Decision and Control (CDC)},
  year={2023},
  organization={IEEE}
}

@inproceedings{simchowitz2019learning,
  title={Learning linear dynamical systems with semi-parametric least squares},
  author={Simchowitz, Max and Boczar, Ross and Recht, Benjamin},
  booktitle={Conference on Learning Theory},
  pages={2714--2802},
  year={2019},
  organization={PMLR}
}

@article{sarkar2021finite,
  title={Finite time LTI system identification},
  author={Sarkar, Tuhin and Rakhlin, Alexander and Dahleh, Munther A},
  journal={Journal of Machine Learning Research},
  volume={22},
  number={26},
  pages={1--61},
  year={2021}
}

@article{meng2016microgrid,
  title={Microgrid supervisory controllers and energy management systems: A literature review},
  author={Meng, Lexuan and Sanseverino, Eleonora Riva and Luna, Adriana and Dragicevic, Tomislav and Vasquez, Juan C and Guerrero, Josep M},
  journal={Renewable and Sustainable Energy Reviews},
  volume={60},
  pages={1263--1273},
  year={2016},
  publisher={Elsevier}
}

@article{bin2021hysteresis,
  title={Hysteresis-based supervisory control with application to non-pharmaceutical containment of COVID-19},
  author={Bin, Michelangelo and Crisostomi, Emanuele and Ferraro, Pietro and Murray-Smith, Roderick and Parisini, Thomas and Shorten, Robert and Stein, Sebastian},
  journal={Annual reviews in control},
  volume={52},
  pages={508--522},
  year={2021},
  publisher={Elsevier}
}

@article{aguiar2007trajectory,
  title={Trajectory-tracking and path-following of underactuated autonomous vehicles with parametric modeling uncertainty},
  author={Aguiar, A Pedro and Hespanha, Joao P},
  journal={IEEE transactions on automatic control},
  volume={52},
  number={8},
  pages={1362--1379},
  year={2007},
  publisher={IEEE}
}

@article{morse1997supervisory,
  author={Morse, A. Stephen},
  title={Supervisory control of families of linear set-point controllers. 2. robustness},
  journal={IEEE Transactions on Automatic Control},
  volume={42},
  number={11},
  pages={1500--1515},
  year={1997},
  publisher={IEEE}
}

@article{anderson2000multiple,
  title={Multiple model adaptive control. Part 1: Finite controller coverings},
  author={Anderson, Brian DO and Brinsmead, Thomas S and De Bruyne, Franky and Hespanha, Joao and Liberzon, Daniel and Morse, A Stephen},
  journal={Int. J. Robust Nonlinear Control},
  year={2000},
  publisher={Wiley Online Library}
}

@incollection{ljung1998system,
  title={System identification},
  author={Ljung, Lennart},
  booktitle={Signal analysis and prediction},
  pages={163--173},
  year={1998},
  publisher={Springer}
}

@book{liberzon2003switching,
  title={Switching in systems and control},
  author={Liberzon, Daniel},
  volume={190},
  publisher={Springer},
  year={2003}
}

@article{morse1996supervisory,
  title={Supervisory control of families of linear set-point controllers-part i. exact matching},
  author={Morse, A Stephen},
  journal={IEEE transactions on Automatic Control},
  volume={41},
  number={10},
  pages={1413--1431},
  year={1996},
  publisher={IEEE}
}

@article{shamma2002gain,
  title={Gain scheduling: Potential hazards and possible remedies},
  author={Shamma, Jeff S and Athans, Michael},
  journal={IEEE Control Systems Magazine},
  volume={12},
  number={3},
  pages={101--107},
  year={2002},
  publisher={IEEE}
}

@article{rugh2000research,
  title={Research on gain scheduling},
  author={Rugh, Wilson J and Shamma, Jeff S},
  journal={Automatica},
  volume={36},
  number={10},
  pages={1401--1425},
  year={2000},
  publisher={Elsevier}
}

@book{wu1995control,
  title={Control of linear parameter varying systems},
  author={Wu, Fen},
  year={1995},
  publisher={University of California, Berkeley}
}

@book{mohammadpour2012control,
  title={Control of linear parameter varying systems with applications},
  author={Mohammadpour, Javad and Scherer, Carsten W},
  year={2012},
  publisher={Springer Science \& Business Media}
}

@inproceedings{chen2021black,
  title={Black-box control for linear dynamical systems},
  author={Chen, Xinyi and Hazan, Elad},
  booktitle={Conference on Learning Theory},
  pages={1114--1143},
  year={2021},
  organization={PMLR}
}

@inproceedings{bakshi2023new,
  title={A new approach to learning linear dynamical systems},
  author={Bakshi, Ainesh and Liu, Allen and Moitra, Ankur and Yau, Morris},
  booktitle={Proceedings of the 55th Annual ACM Symposium on Theory of Computing},
  pages={335--348},
  year={2023}
}

@article{kim2025online,
  title={Online Bandit Nonlinear Control with Dynamic Batch Length and Adaptive Learning Rate},
  author={Kim, Jihun and Lavaei, Javad},
  journal={Transactions on Machine Learning Research},
  year={2025},
  publisher={TMLR}
}

@inproceedings{muehlebach2026the,
    title={The Sample Complexity of Online Reinforcement Learning: A Multi-model Perspective},
    author={Michael Muehlebach and Zhiyu He and Michael I. Jordan},
    booktitle={International Conference on Learning Representations},
    year={2026},
}

@inproceedings{li2024learning,
  title={Learning the uncertainty sets of linear control systems via set membership: A non-asymptotic analysis},
  author={Li, Yingying and Yu, Jing and Conger, Lauren and Kargin, Taylan and Wierman, Adam},
  booktitle={International Conference on Machine Learning},
  year={2024}
}

@inproceedings{sun2024least,
  title={A least-square method for non-asymptotic identification in linear switching control},
  author={Sun, Haoyuan and Jadbabaie, Ali},
  booktitle={2024 IEEE 63rd Conference on Decision and Control (CDC)},
  pages={2993--2998},
  year={2024},
  organization={IEEE}
}

\appendix

\section{Mathematical Preliminaries}
This section several technical results that will become useful in our proofs.

In this work, we employ system identification method based on ordinary least square (OLS).
Consider a linear model
\[ y_t = \Theta_* z_t + r_t, \, t = 1, \dots, T, \]
where the sequence of random vectors $\{y_t\}_{t=1}^T$ and $\{z_t\}_{t=1}^T$ are adapted to a filtration $\{\mathcal{F}_t\}_{t \ge 0}$ and $r_t$ are residuals/noise that are $\sigma^2$-sub-Gaussian.
Then, OLS seeks to recover the true parameter $\Theta_\star$ through the following estimate:
\[\hat\Theta = \left(\sum_{t=1}^{T} y_t z_t^\top\right) V^{-1}, \text{where } V= \sum_t z_t z_t^\top.\]
We can bound the estimation error $\hat\Theta - \Theta_\star$ with the self-normalized martingale tail bound.

\begin{proposition}[Theorem 1 in \cite{abbasi2011improved}]
\label{thm:self-normalized}
    Consider $\{z_t\}_{t=1}^T$ adapted to a filtration $\{\mathcal{F}_t\}_{t \ge 0}$.
    Let $V = \sum_{t=1}^T z_t z_t^\top$.
    If the scalar-valued random variable $r_t \mid \mathcal{F}_{t-1}$ is $\sigma^2$-sub-Gaussian, then for any $V_0 \succeq 0$,
    \[ \norm{\sum_{t=1}^T z_t r_t}_{(V + V_0)^{-1}}^2 \le 2 \sigma^2 \log\left(\delta^{-1} \frac{\mathrm{det}(V + V_0)^{-1/2}}{\mathrm{det}( V_0)^{1/2}}\right)\]
    with probability $1-\delta$.
\end{proposition}

We also need a concentration inequality over a quadratic function on Gaussian random variables (e.g. chi-squared distributions).
\begin{proposition}
\label{thm:hanson-wright}
    Consider a symmetric matrix $M \in \mathbb{S}^{d \times d}$ and a random vector $g \sim \mathcal{N}(0, I_{d \times d})$.
    Then, for any $\delta \in (0, 1/e)$, we have
    \begin{align*}
    \Pr\left(g^\top M g > \tr(M) + 2 \norm{M}_F \sqrt{\log(1/\delta)} + 2\opnorm{M} \log(1/\delta) \right) &< \delta, \\
    \Pr\left(g^\top M g< \tr(M) - 2\norm{M}_F \sqrt{\log(1/\delta)}\right) &< \delta
\end{align*}
\end{proposition}
We note this is a simplified version of the Hanson-Wright inequality~\cite{rudelson2013hanson}, and this particular can be derive from Lemma 1 of~\cite{laurent2000adaptive} and Proposition 1 of~\cite{hsu2012tail}.

Next, we employ the following concentration inequality on martingales, which is crucial to the proof of our main result because the outcomes of different episodes are not independent.
\begin{proposition}[Corollary 2.20 in \cite{wainwright2019high}]
\label{thm:azuma-hoeffding}
Let $\{D_k\}_{k=1}^\infty$ be a martingale difference sequence adapted to a filtration $\{\mathcal{F}_k\}_{k=1}^\infty$.
If $D_k \in [0, 1]$ almost surely for all $k$, then for any $t > 0$,
\[\Pr\left(\frac{1}{n}\left|\sum_{k=1}^n D_k\right| \ge t\right) \le 2 \exp(-2nt^2).\]
\end{proposition}

Finally, we introduce the following technical lemma to manipulate inequalities of the form $b \ge a \log b$, which shows up when we bound the episode length $\tau$. 
\begin{lemma}
\label{thm:nlogn-bound}
    For any positive real numbers $a, b$.
    If $b \ge 2a\log(2a)$, then we have $b \ge a \log b$.
\end{lemma}

\begin{proof}
    We consider two cases on the value of $a$:
    \begin{itemize}
        \item Case 1: $a < e$. Then, we note that $\min_b (b - a\log b) = a - a \log a$, which is positive when $a < e$.
        \item Case 2: $a > e$. In this case, $2a \log(2a) > a$.
        So, for $b \ge 2a \log(2a)$, we have $\frac{\partial}{\partial b} b - a \log b = 1 - a/b > 0$.
        Hence it suffices to check that $b \ge a \log b$ when $b = 2a\log(2a)$.
        \begin{align*}
            2a \log(2a) - a \log(2a \log(2a))
            &= 2a \log(2a) - a \log(2a) - a \log\log(2a) \\
            &= a \log(2a) - a \log\log(2a) \\
            &\ge a \min_x (x - \log x) = a.
        \end{align*}
        So we are done.
    \end{itemize}
\end{proof} 
\section{Proof of Proposition~\ref{thm:instability}}
\label{sec:proof-instability}
In addition to proving Proposition~\ref{thm:instability}, we also present and analyze a variant of Criterion~\ref{alg:instability} that achieves a dimension-free guarantee.
In this variant, we consider unit vectors $u_1, \dots, u_N$ that are associated with the unstable modes of the destabilizing controllers.
The explicit form of these unit vectors will be defined later in the proof.

\begin{algcriteria}
\begin{algorithmic}[1]
    \State Input: an observed trajectory $y_1, \dots, y_\tau$ of length $\tau$, active controller $j$.
    \State Let $\mathcal{O}_{[j,j]}$ be the observability matrix of the model $(C_{[j,j]}, A_{[j,j]}, B_{[j,j]})$, compute
    \[\hat{x}_1 = \mathcal{O}_{[j,j]}^\dagger [y_\nu, \dots, y_1]^\top.\]
    \State Compute a predicted trajectory
    \[\hat{y}_{t} = C_{[j,j]} A_{[j,j]}^{t-1} \hat{x}_1.\]
    \ForAll{unit vectors $u_k, k = 1, \dots, N$}
    \If{$\norm{y_{\tau} - \hat{y}_{\tau}} > \sqrt{2\Theta \log(2n/\delta)}$} \Return 0
\EndIf
    \EndFor
    \State \Return 1
\end{algorithmic}
\caption{Instability detection criterion $S^{(1)}$ (dimension-free variant)}
\label{alg:instability-dim-free}
\end{algcriteria}
This variant achieves the same guarantee in Proposition~\ref{thm:instability}, but will a trajectory length of
\[\tau \gtrsim \nu + \log(1+\varepsilon_a)^{-1} \log\left(\frac{\Theta \log(2n/\delta)}{\varepsilon_a \varepsilon_c \sigma_w}\right).\]

To start with the proof, we write out the following input-output relation for the first $\nu$ steps:
\begin{equation}
\label{equ:obs-dynamics-app}
\begin{bmatrix} y_\nu \\ \vdots \\y_2 \\ y_1 \end{bmatrix}
=
\mathcal{O}_{[i_\star, j]} x_1
+
\tbo{\mathcal{T}^{(\nu-1)}_{[i_\star, j]}}{\mathbf{0}_{1 \times (\nu-1)}}
\begin{bmatrix} w_{\nu-1} \\ \vdots \\w_2 \\ w_1  \end{bmatrix}
+
\tbo{\mathcal{T}^{(\nu-1)}_{[i_\star, j]}}{\mathbf{0}_{1 \times (\nu-1)}}
\begin{bmatrix} B_{[i_\star, j]}u_{\nu-1} \\ \vdots \\ B_{[i_\star, j]}u_2 \\ B_{[i_\star, j]}u_1  \end{bmatrix}
+
\begin{bmatrix} \eta_\nu \\ \vdots \\ \eta_2 \\ \eta_1  \end{bmatrix},
\end{equation}
where we denote $\mathcal{O} = [CA^{\nu-1}; \dots; CA; C]$ as the observability matrix and
\[\mathcal{T}^{(k)} = \begin{bmatrix} 
    C & CA & CA^2 & \dots & CA^{k-1} \\
    0 & C & CA & \dots & CA^{k-2} \\
    0 & 0 & C & \dots & CA^{k-3}\\
    \vdots & \vdots & \vdots & \ddots & \vdots \\
    0 & 0 & 0 & \dots & C
\end{bmatrix}\]
as a $k \times k$ Toeplitz matrix.

\paragraph{Case 1: matching controller ($j = i_\star$)}
Since $\mathcal{O}_{[i_\star, j]} = \mathcal{O}_{[j,j]}$, we can multiply both sides of \eqref{equ:obs-dynamics-app} by the pseudo-inverse $\mathcal{O}_{[j,j]}$ to get
\begin{equation*}
\hat{x}_1 = \mathcal{O}_{[j, j]}^\dagger
\begin{bmatrix} y_\nu \\ \vdots \\y_2 \\ y_1 \end{bmatrix}
=
x_1
+
\mathcal{O}_{[j, j]}^\dagger
\tbo{\mathcal{T}^{(\nu-1)}_{[j, j]}}{\mathbf{0}_{1 \times (\nu-1)}}
\begin{bmatrix} w_{\nu-1} \\ \vdots \\w_2 \\ w_1  \end{bmatrix}
+
\mathcal{O}_{[j, j]}^\dagger
\tbo{\mathcal{T}^{(\nu-1)}_{[j, j]}}{\mathbf{0}_{1 \times (\nu-1)}}
\begin{bmatrix} B_{[j, j]}u_{\nu-1} \\ \vdots \\ B_{[j, j]}u_2 \\ B_{[j, j]}u_1  \end{bmatrix}
+
\mathcal{O}_{[i_\star, j]}^\dagger \begin{bmatrix} \eta_\nu \\ \vdots \\ \eta_2 \\ \eta_1  \end{bmatrix}.
\end{equation*}
Under Assumption~\ref{assum:observability}, the closed-loop system is $\nu$-strictly observable, and we have $\opnorm{\mathcal{O}_{[i_\star, j]}^\dagger} \le 1/\varepsilon_c$.
And since the matching controller is stabilizing, the Toeplitz matrices are bounded by $\hinfnorm{C_{[j,j]}, A_{[j,j]}, B_{[j,j]}}$.
It follows that $x_1 - \hat{x}_1$ is 
\begin{equation}
\label{equ:init-est-error}
    \varepsilon_c^{-1}\left(\hinfnorm{C_{[j,j]}, A_{[j,j]}}\sigma_w^2 + \hinfnorm{C_{[j,j]}, A_{[j,j]}, B_{[j,j]}}\sigma_u^2 + \sigma_\eta^2\right)
\end{equation}
sub-Gaussian.

Then, we recursively expand the system dynamics to have
\begin{align*}
    y_t
    &= C_{[i_\star, j]} A_{[i_\star, j]}^{t-1} x_{1} + \sum_{s=1}^{t-1} C_{[i_\star, j]} A_{[i_\star, j]}^{s-1} B_{[i_\star, j]} u_{t-s} + \sum_{s=1}^{t-1} C_{[i_\star, j]} A_{[i_\star, j]}^{s-1} w_{t-s} + \eta_t.
\end{align*}
which means
\begin{align*} 
    y_{\tau} - \hat{y}_{\tau} = C_{[i_\star, j]} A_{[i_\star, j]}^{\tau-1} (x_{1} - \hat{x}_1) + \sum_{s=1}^{\tau-1} C_{[i_\star, j]} A_{[i_\star, j]}^{s-1} B_{[i_\star, j]} u_{\tau-s} + \sum_{s=1}^{\tau-1} C_{[i_\star, j]} A_{[i_\star, j]}^{s-1} w_{\tau-s} + \eta_\tau.
\end{align*}
Therefore, the quantity $y_{\tau} - \hat{y}_{\tau}$ is 
\begin{align*} 
    &\hinfnorm{C_{[j,j]}, A_{[j,j]}} \varepsilon_c^{-1}\left(\hinfnorm{C_{[j,j]}, A_{[j,j]}}\sigma_w^2 + \hinfnorm{C_{[j,j]}, A_{[j,j]}, B_{[j,j]}}\sigma_u^2 + \sigma_\eta^2\right) \\
    &\quad\quad+ \hinfnorm{C_{[j,j]}, A_{[j,j]}}\sigma_w^2 + \hinfnorm{C_{[j,j]}, A_{[j,j]}, B_{[j,j]}}\sigma_u^2 + \sigma_\eta^2
\end{align*}
sub-Gaussian.
If we let 
\[\zeta = \max \left\{\hinfnorm{C_{[i, i]}, A_{[i, i]}}, \hinfnorm{C_{[i, i]}, A_{[i, i]}, B_{[i, i]}} : 1 \le i \le N\right\},\]
we can say that $y_{\tau} - \hat{y}_{\tau}$ is
$\Theta:= (1 + \zeta \varepsilon_c^{-1})(\zeta \sigma_w^2 + \zeta \sigma_u^2 + \sigma_\eta^2)$
sub-Gaussian.
By the property of sub-Gaussian (see Equ. 2.9 in~\cite{wainwright2019high}), we have that for any unit vector $u$ and $\lambda > 0$, 
\begin{equation}
\label{equ:subgaussian-conc}
    \Pr(|u^\top(y_{\tau} - \hat{y}_{\tau})| \ge \lambda) \le 2 \exp(-\lambda^2/(2\Theta)).
\end{equation}

\paragraph{Case 2: active controller $j$ is destabilizing}
Like in the previous case, we have that
\begin{align*} 
    y_{\tau} - \hat{y}_{\tau} = C_{[i_\star, j]} A_{[i_\star, j]}^{\tau-1} (x_{1} - \hat{x}_1) + \sum_{s=1}^{t-1} C_{[i_\star, j]} A_{[i_\star, j]}^{s-1} B_{[i_\star, j]} u_{\tau-s} + \sum_{s=1}^{t-1} C_{[i_\star, j]} A_{[i_\star, j]}^{s-1} w_{\tau-s} + \eta_\tau.
\end{align*}
Because $\hat{x}_1$ is independent of the noise at time $\nu+1$, we can express the the difference as
\[\Delta := y_{\tau} - \hat{y}_{\tau} = C_{[i_\star, j]} A_{[i_\star, j]}^{\tau-\nu-2} w_{\nu+1} + \xi,\]
where $\xi$ is a quantity independent of $w_{\nu+1}$.

Under Assumption~\ref{assum:explosive}, the closed-loop dynamics $A_{[i_\star, j]}$ has spectral radius $\ge 1+\varepsilon_a$.
So, we let unit vector $q_{i_\star}$ be an eigenvector of $A_{[i_\star, j]}$ corresponding to an eigenvector $\lambda_{[i_\star,j]}$ whose magnitude $\ge 1+\varepsilon_a$.
Then, due to strict observability, we have that
\begin{align*}
    \varepsilon_c \le \norm{\mathcal{O}_{[i_\star, j]} q_{i_\star}}
    &= \norm{\left[\lambda_{[i_\star,j]}^{\nu-1} C_{[i_\star, j]} q_{i_\star}; \dots; \lambda_{[i_\star,j]} C_{[i_\star, j]} q_{i_\star}; C_{[i_\star, j]} q_{i_\star}\right]} \\
    &= \norm{C_{[i_\star, j]}q_{i_\star}} \frac{|\lambda_{[i_\star,j]}|^\nu - 1}{|\lambda_{[i_\star,j]}| - 1} \\
    \implies
    \norm{C_{[i_\star, j]}q_{i_\star}}
    &\ge \frac{\varepsilon_c \varepsilon_a}{|\lambda_{[i_\star,j]}|^\nu}.
\end{align*}
It follows that \[\norm{C_{[i_\star, j]} A_{[i_\star, j]}^{\tau-\nu-2} q_{i_\star}} \ge |\lambda_{[i_\star,j]}|^{\tau-\nu-2}\frac{\varepsilon_c \varepsilon_a}{|\lambda_{[i_\star,j]}|^\nu} \ge \varepsilon_c \varepsilon_a (1+\varepsilon_a)^{\tau-2\nu-2}. \]

Now, we can let $(u_{i_\star}, v_{i_\star})$ be the leading singular vectors of $C_{[i_\star, j]} A_{[i_\star, j]}^{\tau-\nu-2}$ corresponding to the top singular value $\ge \varepsilon_c \varepsilon_a (1+\varepsilon_a)^{\tau-2\nu-2}$.
Since $w_{\nu+1}$ is Gaussian, we know that from standard normal CDF, $\Pr(|v_{i_\star}^\top w_{\nu+1}| > \sigma_w/4) > 4/5$.
Then we have that
\begin{align*}
    \norm{\Delta} \ge |u_{i_\star}^\top \Delta|
    &\ge |u_{i_\star}^\top C_{[i_\star, j]} A_{[i_\star, j]}^{\tau-\nu-2} w_{\nu+1}| \\
    &= \opnorm{C_{[i_\star, j]} A_{[i_\star, j]}^{\tau-\nu-2}} |v_{i_\star}^\top w_{\nu+1}| \\
    &\ge \varepsilon_c \varepsilon_a (1+\varepsilon_a)^{\tau-2\nu-2} \cdot |v_{i_\star}^\top w_{\nu+1}| \\
    &\ge \frac{1}{4}\varepsilon_c \varepsilon_a (1+\varepsilon_a)^{\tau-2\nu-2} \sigma_w \quad \text{w.p. } 2/5.
\end{align*}
Note the probability is halved because by symmetry, there is 50/50 chance that the signs of $u_{i_\star}^\top\xi$ and $u_{i_\star}^\top C_{[i_\star, j]} A_{[i_\star, j]}^{\tau-\nu-2} w_{\nu+1}$ are the same.

We note that for the current active controller $j$ we can pre-compute $(u_{i}, v_{i})$ for $1 \le i \le N$ that makes $A_{[i,j]}$ is unstable (and set the stable ones to 0) by applying an iterative SVD algorithm\footnote{This is typically a modified Lanczos algorithm, such as MATLAB's \texttt{svds} function.} on the matrix $C_{[i,j]} (A_{[i,j]} / \lambda_{[i,j]})^{\tau-\nu-2}$.

Then, we use these $u_i$'s for the threshold in Criterion~\ref{alg:instability-dim-free}, so that for at least one of $1 \le i \le N$, we have $|u_i^\top \Delta| \ge \frac{1}{4}\varepsilon_c \varepsilon_a (1+\varepsilon_a)^{\tau-2\nu-2} \sigma_w \text{ w.p. } 2/5$.

With the analysis for both cases, we can now complete the proof for both variants.

\paragraph{Variant 1 (Criterion~\ref{alg:instability}).}
Note that we can bound $\norm{\Delta}$ by considering the basis vector $e_1, \dots, e_{d_y}$ in $\RR^{d_y}$ and measure $|e_k^\top \Delta|$ for all $k$'s.
Then, after applying union bound to \eqref{equ:subgaussian-conc}, we have that 
\[\Pr\left(|e_k^\top(y_{\tau} - \hat{y}_{\tau})| \le \sqrt{2\Theta\log(2d_y/\delta)} \quad  \forall k\right) \ge 1 - \delta.\]
It follows that
\[\Pr\left(\norm{y_{\tau} - \hat{y}_{\tau}} \le \sqrt{2d_y\Theta\log(2d_y/\delta)}\right) \ge 1 - \delta.\]
So, this criterion correctly finds the matching controller to be stabilizing with probability $1-\delta$.
Next, if we set
\[\tau = 2 + 2 \nu + \log(1+\varepsilon_a)^{-1} \log\left(\frac{4\sqrt{2d_y\Theta\log(2d_y/\delta)}}{\varepsilon_a \varepsilon_c \sigma_w}\right),\]
an unstable $A_{[i_\star, j]}$ results in $\norm{y_{\tau} - \hat{y}_{\tau}} \le \sqrt{2d_y\Theta\log(2d_y/\delta)}$ with probability $\ge 2/5$.
So, this criteria correctly detect an unstable $A_{[i_\star, j]}$ with probability at least 2/5.

\paragraph{Variant 2 (Criterion~\ref{alg:instability-dim-free}).}
After applying union bound to \eqref{equ:subgaussian-conc}, we have that 
\[\Pr\left(|u_k^\top(y_{\tau} - \hat{y}_{\tau})| \le \sqrt{2\Theta\log(2n/\delta)} \quad  \forall k\right) \ge 1 - \delta.\]
So, this criterion correctly finds the matching controller to be stabilizing with probability $1-\delta$.
Next, if we set
\[\tau = 2 + 2 \nu + \log(1+\varepsilon_a)^{-1} \log\left(\frac{4\sqrt{2\Theta\log(2n/\delta)}}{\varepsilon_a \varepsilon_c \sigma_w}\right),\]
an unstable $A_{[i_\star, j]}$ results in $u_{i_\star}^\top(y_{\tau} - \hat{y}_{\tau})| \le \sqrt{2d_y\Theta\log(2d_y/\delta)}$ with probability $\ge 2/5$.
So, this criteria correctly detect an unstable $A_{[i_\star, j]}$ with probability at least 2/5. 
\section{Proof of Proposition~\ref{thm:sys-id}}
\label{sec:proof-sysid}
Recall that, in Criterion~\ref{alg:sys-id}, we defined 
\[\widetilde{y}_{t} = y_t - C_{[j,j]} A_{[j,j]}^{t-1} \hat{x}_1 - G_{[j,j]} z_t = (G_{[i_\star, j]} - G_{[j, j]}) z_t + r_t.\]
Assumption~\ref{assum:markov-sep} implies that, for every $k \neq j$, there exist \textit{critical directions} $(u_k, v_k)$ satisfying $u_k^\top (G_{[k, j]} - G_{[j, j]}) v_k = \opnorm{G_{[k, j]} - G_{[j, j]}} \ge \gamma$.
Next, we estimate the difference $G_{[i_\star, j]} - G_{[j, j]}$ along these critical directions by performing the OLS estimate from data points $\left\{(v_k^\top z_t, u_k^\top \widetilde{y}_t)\right\}_{t=\nu+h+1}^\tau$.
In particular, we define the OLS estimate along a direction $(u, v)$ as
\begin{align*}
    \Delta_{u, v}
    &= \left(\sum_{t=\nu+h+1}^{\tau} (u^\top \widetilde{y}_t) (v^\top z_t)\right) \left(\sum_{t=\nu+h+1}^{\tau} (v^\top z_t)^2\right)^{-1} \\
    &= \left(\sum_{t=\nu+h+1}^{\tau} u^\top ((G_{[i_\star, j]} - G_{[j, j]})z_t + r_t) (v^\top z_t)\right) \left(\sum_{t=\nu+h+1}^{\tau} (v^\top z_t)^2\right)^{-1}.
\end{align*}
Then, this criteria considers the active controller's associated model as matching the true system if $\Delta_{u_k, v_k}$ is small along every critical direction.
And this criteria declares non-matching if any one of the $\Delta_{u_k, v_k}$ is large.

\paragraph{Case 1: matching controller ($j = i_\star$).}
It follows that for each critical direction, we have
\begin{align*}
    \Delta_{u_j, v_j}
    &= \left(\sum_{t=\nu+h+1}^{\tau} (u^\top \widetilde{y}_t) (v^\top z_t)\right) \left(\sum_{t=\nu+h+1}^{\tau} (v^\top z_t)^2\right)^{-1} \\
    &= \left(\sum_{t=\nu+h+1}^{\tau}  (u_j r_t) (v_j^\top z_t)\right) \left(\sum_{t=\nu+h+1}^{\tau} (v^\top z_t)^2\right)^{-1}.
\end{align*}
Therefore, it suffices to show that for arbitrary unit vectors $(u, v)$, the quantity 
\begin{equation}
\label{equ:ols-uv}
    \Delta_{u, v} = \left(\sum_{t=\nu+h+1}^{\tau} (u^\top r_t)(v^\top z_t)\right) \underbrace{\left(\sum_{t=\nu+h+1}^{\tau} (v^\top z_t)^2 \right)^{-1}}_{:= \Phi^{-1}}
\end{equation}
satisfies $|\Delta_{u,v}| \le \gamma$ with probability $1-\delta/N$.
Then, from union bound, Criterion~\ref{alg:sys-id} correctly finds controller $j$ to match the true system with probability $1-\delta$.

To begin, we first establish that the residual term $r_t$ is sub-Gaussian.
In Section~\ref{sec:sys-id}, we derived that
\begin{align*}
    y_t
    &= C_{[i_\star, j]} A_{[i_\star, j]}^{h} x_{t-h} +  G_{[i_\star, j]} z_t +  \sum_{s=1}^h C_{[i_\star, j]} A_{[i_\star, j]}^{s} w_{t-s-1} +  \eta_t.
\end{align*}
And since
\[x_{t-h} = A_{[i_\star, j]}^{t-h-1} x_{1} + \sum_{s=1}^{t-h-1} A_{[i_\star, j]}^{s-1} B_{[i_\star, j]} u_{t-h-s} + \sum_{s=1}^{t-h-1} A_{[i_\star, j]}^{s-1} w_{t-h-s},\]
we can write
\begin{align*}
    y_t
    &= C_{[i_\star, j]} A_{[i_\star, j]}^{t-1} x_{1} + G_{[i_\star, j]} z_t + \sum_{s=1}^{t-h-1} C_{[i_\star, j]} A_{[i_\star, j]}^{s+h-1} B_{[i_\star, j]} u_{t-h-s} + \sum_{s=1}^{t-1} C_{[i_\star, j]} A_{[i_\star, j]}^{s-1} w_{t-s} + \eta_t.
\end{align*}
Therefore, 
\begin{align*}
    r_t 
    &= y_t - C_{[i_\star,j]}A_{[i_\star,j]}^{t-1} \hat{x}_1 - G_{[i_\star,j]} z_t \\
    &= C_{[i_\star, j]} A_{[i_\star, j]}^{t-1} (x_{1} - \hat{x}_1) + \sum_{s=1}^{t-h-1} C_{[i_\star, j]} A_{[i_\star, j]}^{s+h-1} B_{[i_\star, j]} u_{t-h-s} + \sum_{s=1}^{t-1} C_{[i_\star, j]} A_{[i_\star, j]}^{s-1} w_{t-s} + \eta_t.
\end{align*}
Recall that, from \eqref{equ:init-est-error} in the previous part, the estimation error on the initial state is $x_1 - \hat{x}_1$ is 
\[ \varepsilon_c^{-1}(\hinfnorm{C_{[i_\star, i_\star]}, A_{[i_\star, i_\star]}}\sigma_w^2 + \hinfnorm{C_{[i_\star, i_\star]}, A_{[i_\star, i_\star]}, B_{[i_\star, i_\star]}}\sigma_u^2 + \sigma_\eta^2)\]
sub-Gaussian.
We conclude that $r_t$ is
\begin{align*}
    & \hinfnorm{C_{[i_\star, i_\star]}, A_{[i_\star, i_\star]}} \varepsilon_c^{-1}(\hinfnorm{C_{[i_\star, i_\star]}, A_{[i_\star, i_\star]}}\sigma_w^2 + \hinfnorm{C_{[i_\star, i_\star]}, A_{[i_\star, i_\star]}, B_{[i_\star, i_\star]}}\sigma_u^2 + \sigma_\eta^2) \\
    &\quad\quad+ \hinfnorm{C_{[i_\star, i_\star]}, A_{[i_\star, i_\star]}} \sigma_w^2 + \hinfnorm{C_{[i_\star, i_\star]}, A_{[i_\star, i_\star]}, B_{[i_\star, i_\star]}} \sigma_u^2 + \sigma_\eta^2
\end{align*}
sub-Gaussian.
If we let 
\[\zeta = \max \left\{\hinfnorm{C_{[i, i]}, A_{[i, i]}}, \hinfnorm{C_{[i, i]}, A_{[i, i]}, B_{[i, i]}} : 1 \le i \le N\right\},\]
we can say that $r_t$ is
$\sigma_r^2 := (1 + \zeta \varepsilon_c^{-1})(\zeta \sigma_w^2 + \zeta \sigma_u^2 + \sigma_\eta^2)$
sub-Gaussian.
Additionally, we note that this calculations implies that $z_t$ and $r_t$ are independent.

Then, applying Proposition~\ref{thm:self-normalized} to the OLS estimate \eqref{equ:ols-uv} yields that for $\mu > 0$, and probability $1-\delta/(4N)$,
\begin{equation}
\label{equ:self-normalized-delta}
    \frac{\Phi^2 \Delta_{u, v}^2}{\Phi + \mu} = \frac{1}{\Phi + \mu} \left(\sum_{t=\nu+h+1}^{\tau} (u^\top r_t)(v^\top z_t)\right)^2 \le 2 \sigma_r^2 \log\left(\frac{4N\sqrt{\Phi + \mu}}{\delta\sqrt{\mu}}\right).
\end{equation}

To bound $\Delta_{u, v}$, we need to find both upper and lower bound for $\Phi$.
Also, for convenience, define $\tau' = \tau - \nu - h$.
For the upper bound, we note that $\Phi$ can be written as a quadratic form $g^T M g$, where $g \in \RR^{\tau d}$ and indexed so that $g_{i,j} = u_i[j]$.
We note that we can write as
\[\sum_{t=\nu+h+1}^\tau (v^\top z_t)^2 = \sum_{t=\nu+h+1}^\tau z_t^\top v v^\top z_t = \sum_{t=\nu+h+1}^\tau g^\top M_t g,\]
where $M_t$ has a submatrix equal to $vv^\top$ and rest of the entries are zero.

Then, by Hanson-Wright inequality (Proposition~\ref{thm:hanson-wright}), we have that for $\lambda \in (0, 1)$,
\begin{align*}
    \Pr\left(\frac{1}{\sigma_u^2} g^\top M g > \tr(M) + 2 \norm{M}_F \sqrt{\lambda} + 2\opnorm{M} \lambda \right) < \Pr\left(\frac{1}{\sigma_u^2} g^\top M g > \tr(M) + 4 \norm{M}_F \sqrt{\lambda}\right) &< \exp(-\lambda), \\
    \Pr\left(\frac{1}{\sigma_u^2} g^\top M g < \tr(M) - 2\norm{M}_F \sqrt{\lambda}\right) &< \exp(-\lambda)
\end{align*}
For the trace, we note that $\tr(M) = \sum_t \tr(M_t) = \tau'$.
And the Frobenius norm satisfies $\norm{M}_F  \le \sum_t \norm{M_t}_F = \tau'.$
Therefore
\begin{align*} 
\Pr\left(\Phi \ge \sigma_u^2 (\tau' + 4 \sqrt{\tau' \lambda})\right) \le \exp(-\lambda), \\
\Pr\left(\Phi < \sigma_u^2 (\tau' - 2 \sqrt{\tau' \lambda})\right) \le \exp(-\lambda).
\end{align*}
Setting $\lambda = \log(4N/\delta)$ in the first line, we have
\begin{equation}
\label{equ:phi-upper-bound}
    \Pr\left(\Phi \ge \sigma_u^2 (3\tau' + 2 \log(4N/\delta))\right) \le \Pr\left(\Phi \ge \sigma_u^2 (\tau' + 4 \sqrt{\tau'\log(4N/\delta)})\right) \le \delta/(4N).
\end{equation}
Setting $\lambda = \tau'/16$ in the second line, we have
\begin{equation}
\label{equ:phi-lower-bound}
    \Pr\left(\Phi < \frac{\sigma_u^2 \tau'}{2}\right) \le \exp\left(-\frac{\tau'}{16}\right).
\end{equation}

With these bounds in mind, we claim that for any $\mu > 0$,
\begin{equation}
\label{equ:tau-three-quantities}
    \tau' = \max\left\{16\log(4N/\delta), \frac{2\mu}{\sigma_u^2}, \frac{8\sigma_r^2}{\sigma_u^2 \gamma^2}\log\left(\frac{800 N^2 \sigma_r^2}{\delta^2\gamma^2\mu}\right)\right\}
\end{equation}
ensures that $\Delta_{u,v} \le \gamma$ with probability $1-\delta/N$.

First, applying the first two quantities to \eqref{equ:phi-lower-bound} yields that $\Phi \ge \mu$ with probability $1-\delta/(4N)$.
Thus, we have that
\begin{equation}
\label{equ:error-upper-bound}
\begin{aligned}
    \Delta_{u, v}^2 
    &\le \frac{\Phi+\mu}{\Phi^2} \cdot 2\sigma_r^2 \log\left(\frac{4N\sqrt{\Phi + \mu}}{\delta\sqrt{\mu}}\right) \\
    &\le \frac{2}{\Phi} \sigma_r^2 \log\left(\frac{32 N^2\Phi}{\delta^2\mu}\right) \\
    &\le \frac{2}{\Phi} \sigma_r^2 \log\left(\frac{32 N^2 \sigma_u^2 (3\tau' + 2 \log(4N/\delta))}{\delta^2\mu}\right) \\
    &\le \frac{2}{\Phi} \sigma_r^2 \log\left(\frac{100 N^2 \sigma_u^2 \tau'}{\delta^2\mu}\right).
\end{aligned}
\end{equation}
with probability $1-3\delta/(4N)$, where the first line follows from \eqref{equ:self-normalized-delta}, the third line follows from \eqref{equ:phi-upper-bound}, and the final line follows from the first quantity of \eqref{equ:tau-three-quantities}.

Also, from Lemma~\ref{thm:nlogn-bound}, the third quantity of \eqref{equ:tau-three-quantities} implies
\begin{align*}
    & \frac{100 N^2 \sigma_u^2}{\delta^2\mu} \tau' \ge \frac{100 N^2 \sigma_u^2}{\delta^2\mu} \cdot \frac{8\sigma_r^2}{\sigma_u^2 \gamma^2} \log\left(\frac{100 N^2 \sigma_u^2}{\delta^2\mu} \cdot \frac{8\sigma_r^2}{\sigma_u^2 \gamma^2}\right) \\
    \Rightarrow{}& \frac{100 N^2 \sigma_u^2}{\delta^2\mu} \tau' \ge \frac{100 N^2 \sigma_u^2}{\delta^2\mu} \cdot \frac{4\sigma_r^2}{\sigma_u^2 \gamma^2}  \log\left(\frac{100 N^2 \sigma_u^2 \tau'}{\delta^2\mu}\right) \\
    \Leftrightarrow{}& \tau' \ge \frac{4\sigma_r^2}{\sigma_u^2 \gamma^2}  \log\left(\frac{100 N^2 \sigma_u^2 \tau'}{\delta^2\mu}\right).
\end{align*}
Then, applying the first quantity of \eqref{equ:tau-three-quantities} to \eqref{equ:phi-lower-bound}, we have that
\begin{equation}
\label{equ:persistency-excitation}
\Phi \ge \frac{\sigma_u^2\tau'}{2} \ge \frac{2}{\gamma^2} \sigma_r^2 \log\left(\frac{100 N^2 \sigma_u^2 \tau'}{\delta^2\mu}\right)
\end{equation}
with probability $1 - \delta/(4N)$.
Therefore, by combining \eqref{equ:error-upper-bound} and \eqref{equ:persistency-excitation}, we conclude that $|\Delta_{u, v}| \le \gamma$ with probability $1-\delta/N$.

Finally, by picking $\mu = \sigma_r^2 / \gamma^2$ and union bounding over all $N$ critical directions, we have that when
\[\tau = \nu + h + \max\left\{16\log(4N/\delta), \frac{2\sigma_r^2}{\sigma_u^2\gamma^2}, \frac{8\sigma_r^2}{\sigma_u^2 \gamma^2}\log\left(\frac{800 N^2}{\delta^2}\right)\right\} \asymp \frac{\sigma_r^2}{\sigma_u^2 \gamma^2} \log(N / \delta),\]
the quantity $\Delta_{u_k, v_k}$ is small along all critical directions.
So, Criterion~\ref{alg:sys-id} correctly outputs 1 with probability $1-\delta$.

\paragraph{Case 2: non-matching controller $j \neq i_\star$.}

Per the definition of the critical directions, we have that $u_{i_\star}^\top(G_{[i_\star, j]} - G_{[j, j]})v_{i_\star} = \opnorm{G_{[i_\star, j]} - G_{[j, j]}} \ge \gamma$.
Then, we have
\begin{align*}
    \Delta_{u_{i_\star}, v_{i_\star}}
    &= \left(\sum_{t=\nu+h+1}^{\tau} u_{i_\star}^\top ((G_{[i_\star, j]} - G_{[j, j]})z_t (v_{i_\star}^\top z_t) + (u_{i_\star}^\top r_t)(v_{i_\star}^\top z_t)\right) \left(\sum_{t=\nu+h+1}^{\tau} (v^\top z_t)^2 \right)^{-1} \\
    &= \left(\sum_{t=\nu+h+1}^{\tau} \opnorm{G_{[i_\star, j]} - G_{[j, j]}}v_{i_\star}^\top z_t (v_{i_\star}^\top z_t) + (u_{i_\star}^\top r_t)(v_{i_\star}^\top z_t)\right) \left(\sum_{t=\nu+h+1}^{\tau} (v^\top z_t)^2 \right)^{-1} \\
    &= \opnorm{G_{[i_\star, j]} - G_{[j, j]}} \left(\sum_{t=\nu+h+1}^{\tau} (v_{i_\star}^\top z_t)^2\right) \left(\sum_{t=\nu+h+1}^{\tau} (v_{i_\star}^\top z_t)^2 \right)^{-1} \\
    &\hspace{8em}+ \left(\sum_{t=\nu+h+1}^{\tau} (u_{i_\star}^\top r_t)(v_{i_\star}^\top z_t)\right) \left(\sum_{t=\nu+h+1}^{\tau} (v_{i_\star}^\top z_t)^2 \right)^{-1} \\
    &= \opnorm{G_{[i_\star, j]} - G_{[j, j]}} + \left(\sum_{t=\nu+h+1}^{\tau} (u_{i_\star}^\top r_t)(v_{i_\star}^\top z_t)\right) \left(\sum_{t=\nu+h+1}^{\tau} (v_{i_\star}^\top z_t)^2 \right)^{-1} \\
    &\ge \gamma + \left(\sum_{t=\nu+h+1}^{\tau} (u_{i_\star}^\top r_t)(v_{i_\star}^\top z_t)\right) \left(\sum_{t=\nu+h+1}^{\tau} (v_{i_\star}^\top z_t)^2 \right)^{-1} \\
    &\ge \gamma \quad \text{w.p. } 1/2,
\end{align*}
where the second line follows from the fact that $(u_{i_\star}, v_{i_\star})$ are the leading singular vectors of $G_{[i_\star, j]} - G_{[j, j]}$, and the last inequality holds because the second term is odd in $z_t$'s, so it is positive with probability 1/2.
Therefore, Criterion~\ref{alg:sys-id} correctly finds controller $j$'s associated model not matching the true system with probability at least 1/2.

\section{Proof of Theorem~\ref{thm:ucb-unified}}
\label{sec:proof-ucb}
We shall prove the two parts of this theorem in order.

\subsection*{Part 1: Sample complexity}
Given Proposition~\ref{thm:instability} and~\ref{thm:sys-id}, we can instantiate the following guarantees for our evaluation metric:
\begin{corollary}
\label{thm:combined-metric}
    Given that each episode has sufficiently length of $\tau = \max(T_1, T_2)$ steps,
    the metric $S(y_1, \dots, y_\tau; j) = \lfloor \frac{1}{2} (S^{(1)}(y_1, \dots, y_\tau; j) + S^{(2)}(y_1, \dots, y_\tau; j)) \rfloor$ satisfies:
    \begin{itemize}[wide, labelindent=3pt]
    \item If the active controller matches the true system ($i_\star = j$), then $S = 1$ with probability $\ge 14/15$.
    \item If the active controller destabilizes the true system, then $S = 0$ with probability $\ge 2/5$.
    \item If the active controller is not matching ($i_\star \neq j$) but stabilizing, then $S = 0$ with probability $\ge 1/2$.
    \end{itemize}
\end{corollary}
Note that these bounds are instantiated by setting $\delta_1=\delta_2=1/30$ for both Propositions.

Recall that for any $1 \le i \le N$ and $q > 0$, we defined the average score to be:
\[\bar{s}_i[q] = \frac{1}{q} \sum_{\ell=1}^q s_i[\ell].\]
We first utilize Martingale concentration inequalities to show that
\begin{itemize}[wide, labelindent=3pt]
    \item If $i = i_\star$, then 
    \[ \Pr\left(\bar{s}_i[q] \le \frac{14}{15} - \sqrt{\frac{\alpha}{q}}\right) \le \exp(-2 \alpha).\]
    \item If $i \neq i_\star$, then 
    \[ \Pr\left(\bar{s}_i[q] \ge \frac{3}{5} + \sqrt{\frac{\alpha}{q}}\right) \le \exp(-2 \alpha).\]
\end{itemize}

To show this claim, we consider each of the three cases in Corollary~\ref{thm:combined-metric} separately.
First, with the matching controller ($i_\star = j$), we denote $\EE_{q-1}[s_{i_\star}[q]]$ as the random variable conditioned on the initial state of the episode where controller $i_\star$ is being applied for the $q$-th time.
Due to Corollary~\ref{thm:combined-metric}, we have that $\EE_{q-1}[s_{i_\star}[q]] \ge 14/15.$
Then, by Azuma-Hoeffding inequality (Proposition~\ref{thm:azuma-hoeffding}), we can bound
\[\Pr\left(\bar{s}_{i_\star}[q] \le \frac{14}{15} - \sqrt{\frac{\alpha}{q}}\right) \le \Pr\left(\frac{1}{q} \sum_{\ell=1}^q \left(s_{i_\star}[\ell] - \EE_{\ell-1}\left[s_{i_\star}[\ell]\right]\right) \le - \sqrt{\frac{\alpha}{q}}\right) \le \exp\left(-2q \left(\sqrt{\alpha/q}\right)^2\right) = \exp(-2\alpha).\]
Similarly, we have that when the active controller $j$ is destabilizing the underlying plant, we have that $\EE_{q-1}[s_j[q]] \le 3/5$.
Then, by Azuma-Hoeffding inequality, we have
\[\Pr\left(\bar{s}_{j}[q] \ge \frac{3}{5} + \sqrt{\frac{\alpha}{q}}\right) \le \Pr\left(\frac{1}{q} \sum_{\ell=1}^q \left(s_j[\ell] - \EE_{\ell-1}\left[s_j[\ell]\right]\right) \ge \sqrt{\frac{\alpha}{q}}\right) \le \exp(-2\alpha).\]
Finally, by the same logic, we have that the remaining controllers that are non-matching ($i_\star \neq j$) but stabilizing satisfies
\[\Pr\left(\bar{s}_{j}[q] \ge \frac{3}{5} + \sqrt{\frac{\alpha}{q}}\right) \le \exp(-2\alpha).\]

Next, having established concentration of the average scores, we consider the two algorithmic variants separately.

\paragraph{Variant 1:} $a_\ell = \frac{L}{72N}$.

We consider the event that
\[\mathcal{E} = \left\{\left(\bar{s}_{i_\star}[q] \ge \frac{14}{15} - \sqrt{\frac{a_\ell}{q}}\right) \wedge \left(\bar{s}_i[q] \le \frac{3}{5} + \sqrt{\frac{a_\ell}{q}}\right) \;\; \forall i \neq i_\star, 1 \le q \le L \right\}.\]
Under this event $\mathcal{E}$, we notice that 
\[\bar{s}_{i_\star}[q] + \sqrt{\frac{a_\ell}{q}} \ge \frac{14}{15} \;\; \forall q,\]
and for any $i \neq i_\star$, if $q > 36 a_\ell$, then
\[\bar{s}_{i}[q] + \sqrt{\frac{a_\ell}{q}} \le \frac{3}{5} + 2 \sqrt{\frac{a_\ell}{q}} < \frac{3}{5} + \frac{1}{3} = \frac{14}{15} \;\; \forall q > 36 a_\ell.\]
Due to the selection rule of Algorithm~\ref{alg:meta}, we conclude that any of the non-matching controller can only been selected at most $36 a_\ell = \frac{L}{2N}$ times.
Therefore, under event $\mathcal{E}$, the matching controller is the most frequently selected one and the algorithm succeeds in identification (i.e. $\hat{i} = i_\star$).

It remains to find sufficiently large number of episodes $L$ so that the event $\mathcal{E}$ occurs with probability $1-\delta$.
From our calculations above, the probability of this event occurring is at least
\begin{align*}
    \Pr(\mathcal{E})
    &\ge 1 - LN \exp(-2 a_\ell) = 1 - LN \exp(-L / 36N).
\end{align*}
Then, according to Lemma~\ref{thm:nlogn-bound}, we have that
\begin{align*}
    &L \ge 72N \log(72N^2/\delta) \\
    \Leftrightarrow{}& \frac{LN}{\delta} \ge \frac{72N^2}{\delta} \log(72N^2/\delta) \\
    \Rightarrow{}& \frac{LN}{\delta} \ge \frac{36N^2}{\delta} \log(LN/\delta) \\
    \Leftrightarrow{}& \frac{L}{36N} \ge \log(LN/\delta) \\
    \Leftrightarrow{}& LN \exp(-L/36N) \le \delta.
\end{align*}
Therefore, with $a_\ell = \frac{L}{72N}$ and $L = 72N \log(72N^2/\delta) \in O(N \log(N/\delta))$, Algorithm~\ref{alg:meta} identifies the matching controller with probability $1-\delta$.

\paragraph{Variant 2:} $a_\ell = \frac{1}{2} \log(\frac{\pi^2 N \ell^2}{6\delta})$.
We consider the event that
\[\mathcal{E} = \left\{\left(\bar{s}_{i_\star}[Q_{i_\star}(\ell)] \ge \frac{14}{15} - \sqrt{\frac{a_\ell}{Q_{i_\star}(\ell)}}\right) \wedge \left(\bar{s}_i[Q_i(\ell)] \le \frac{3}{5} + \sqrt{\frac{a_\ell}{Q_i(\ell)}}\right) \;\; \forall i \neq i_\star, 1 \le \ell \le L \right\}.\]
We note that since $Q_i(\ell) \le \ell$ and $a_\ell$ is increasing in $\ell$, we have that
\[\mathcal{E} \supseteq \bigcap_{q=1}^L \mathcal{E}_\ell, \text{ where } \mathcal{E}_q = \left\{\left(\bar{s}_{i_\star}[q] \ge \frac{14}{15} - \sqrt{\frac{a_q}{q}}\right) \wedge \left(\bar{s}_i[q] \le \frac{3}{5} + \sqrt{\frac{a_q}{q}}\right) \;\; \forall i \neq i_\star \right\}\]
From our calculations above, the probability of each event occurring is at least
\begin{align*}
    \Pr(\mathcal{E}_q)
    &\ge 1 - N \exp(-2 a_q) = 1 - N \exp(-\log(\pi^2q^2 N / 6\delta)) = 1 - \frac{6\delta}{\pi^2 q^2}.
\end{align*}
Then, by union bound, we have that
\begin{align*}
    \Pr(\mathcal{E})
    &\ge 1 - \sum_{q=1}^L \frac{6\delta}{\pi^2 q^2} \ge 1 - \frac{6\delta}{\pi^2} \sum_{q=1}^\infty q^{-2} = 1 - \delta.
\end{align*}

Under this event $\mathcal{E}$, we notice that 
\[\bar{s}_{i_\star}[Q_{i_\star}(\ell)] + \sqrt{\frac{a_\ell}{Q_{i_\star}(\ell)}} \ge \frac{14}{15} \;\; \forall \ell,\]
and for any $i \neq i_\star$, if $Q_i(\ell) > 36 a_\ell$, then
\[\bar{s}_{i}[Q_i(\ell)] + \sqrt{\frac{a_\ell}{Q_i(\ell)}} < \frac{3}{5} + \frac{1}{3} = \frac{14}{15} \;\; \forall \ell: Q_i(\ell) > 9 a_\ell.\]
Due to the selection rule of Algorithm~\ref{alg:meta}, we conclude that any of the non-matching controller can only been selected at most $36 a_L = 18 \log(\frac{\pi^2 N L^2}{6\delta})$ times.
If $L \ge 2 N \cdot 36 a_L$, then under event $\mathcal{E}$, the matching controller is the most frequently selected one and the algorithm succeeds in identification (i.e. $\hat{i} = i_\star$).

Finally, we find a value of $L$ satisfying $L \ge 72N a_L = 36 N \log(\frac{\pi^2 N L^2}{6\delta})$.
Then according to Lemma~\ref{thm:nlogn-bound}, we have that
\begin{align*}
    & L \ge 144 N \log\left(\frac{24\sqrt{6}}{\sqrt{\delta}} \pi N^{3/2}\right) \\
    \Leftrightarrow{}& \sqrt{\frac{N}{6\delta}} \cdot \pi L \ge 144 \frac{N^{3/2}}{\sqrt{6\delta}} \pi \log\left(144 \frac{N^{3/2}}{\sqrt{6\delta}} \pi\right) \\
    \Rightarrow{}& \sqrt{\frac{N}{6\delta}} \cdot \pi L \ge 72 \frac{N^{3/2}}{\sqrt{6\delta}} \pi \log\left(\sqrt{\frac{N}{6\delta}} \cdot \pi L\right) \\
    \Leftrightarrow{}& L \ge 72 N \log\left(\sqrt{\frac{N}{6\delta}} \cdot \pi L\right)\\
    \Leftrightarrow{}& L \ge 36N \log\left(\frac{\pi^2 N L^2}{6\delta}\right).
\end{align*}

Therefore, with $a_\ell = \frac{1}{2} \log(\frac{\pi^2 N \ell^2}{6\delta})$ and $L = 144 N \log\left(\frac{24\sqrt{6}}{\sqrt{\delta}} \pi N^{3/2}\right) \in O(N \log (N/\delta))$, Algorithm~\ref{alg:meta} identifies the matching controller with probability $1-\delta$.

\subsection*{Part 2: System stability}
For this part, we consider the two stages of the algorithm.
In the first $L$ episodes of the algorithm, we are testing different controllers that are potentially destabilizing.
To derive a (conservative) bound for this stage, let us define
\[R_1 := \max_{i, j} \opnorm{A_{[i,j]}}, R_2 := \max_{i, j} \opnorm{B_{[i,j]}}.\]
And for convenience, denote $T' = \tau L$ and $A_t$ be the closed-loop dynamics at time $t$.
Specifically, we have $A_t = A_{[i_\star, i_{\lceil t / \tau \rceil}]}$.

Then, we directly expand the closed-loop dynamics:
\begin{align*}
\begin{bmatrix} x_{T'} \\ \vdots \\x_2 \\ x_1 \end{bmatrix}
=
\mathcal{T}^{(T')}
\begin{bmatrix} w_{T'} \\ \vdots \\w_2 \\ w_1  \end{bmatrix}
+
\tbo{\mathcal{T}^{(T'-1)}}{\mathbf{0}_{1 \times (\nu-1)}}
\begin{bmatrix} B_{T'} u_{T'-1} \\ \vdots \\ B_2 u_2 \\ B_1 u_1  \end{bmatrix},
\end{align*}
where
$$\mathcal{T}^{(k)} = 
\begin{bmatrix} 
    I & A_{T'-1} & A_{T'-1}A_{T'-2} & A_{T'-1}A_{T'-2}A_{T'-3} & \dots &A_{T'-1}A_{T'-2} \dots A_{T'-k+1}\\ 
    0 & I & A_{T'-2} & A_{T'-2}A_{T'-3} & \dots & A_{T'-1}A_{T'-2} \dots A_{T'-k+2} \\ 
    0 & 0 & I & A_{T'-3} & \dots & A_{T'-1}A_{T'-2} \dots A_{T'-k+3}\\ 
    \vdots & \vdots & \vdots & \vdots & \ddots & \vdots
\end{bmatrix}$$
is a $k \times k$ block matrix.
In particular, we can write
\[T^{(k)} = \sum_{\ell=1}^k H^\ell D_\ell,\]
where $H$ is a shift matrix (with 1's on the first super-diagonal and 0's elsewhere), and $D_\ell$ is a block-diagonal matrix corresponding to the diagonals of $\mathcal{T}$.
It follows that 
\[\opnorm{T^{(k)}} \le \sum_{\ell=1}^k R_1^{\ell-1} \le \frac{R_1^{k}}{R_1 -1}.\]
Therefore,
\[\sum_{t=1}^{T'} \norm{x_t}^2 \le 2 \frac{R_1^{2T'}}{(R_1-1)^2} \sum_{t=1}^{T'} \norm{w_t}^2 + 2 \frac{R_1^{2T'-2}R_2}{(R_1-1)^2} \sum_{t=1}^{T'} \norm{u_t}^2. \]

Now, when the UCB portion of the algorithm succeeds in identifying the matching controller, the closed-loop system follows a stable trajectory after time $T'$.
Specifically, for constants $\kappa_0, \kappa_1$ associated with the matched closed-loop dynamics $A_{[i_\star, i_\star]}$, we have
\[\sum_{t=T'+1}^{T} \norm{x_t}^2 \le \kappa_0 \norm{x_{T'}}^2 + \kappa_1 \sum_{t=T'+1}^{T} (\norm{w_t}^2 + \norm{u_t}^2).\]

In conclusion, we can choose
\[C_0 = 2 \kappa_0 \max\left\{\frac{R_1^{2T'}}{(R_1-1)^2}, \frac{R_1^{2T'-2}R_2}{(R_1-1)^2}\right\}, C_1 = \kappa_1.\] 
\end{document}